\documentclass{amsart}

\usepackage{amsmath}
\usepackage{amssymb}
\usepackage{paralist}
\usepackage[english]{babel}
\usepackage[colorlinks=true]{hyperref}
\usepackage{enumitem}
\usepackage{array}
\usepackage[latin1]{inputenc}
\usepackage{xcolor}
\usepackage{esint}
\usepackage{latexsym,amsfonts}
\usepackage{amsthm}
\usepackage{cleveref}
\usepackage{mathtools}
\usepackage{verbatim}
\usepackage{graphics}

\def\vint_#1{\mathchoice
          {\mathop{\kern 0.2em\vrule width 0.6em height 0.69678ex depth -0.58065ex
                  \kern -0.8em \intop}\nolimits_{\kern -0.4em#1}}
          {\mathop{\kern 0.1em\vrule width 0.5em height 0.69678ex depth -0.60387ex
                  \kern -0.6em \intop}\nolimits_{#1}}
          {\mathop{\kern 0.1em\vrule width 0.5em height 0.69678ex
              depth -0.60387ex
                  \kern -0.6em \intop}\nolimits_{#1}}
          {\mathop{\kern 0.1em\vrule width 0.5em height 0.69678ex depth -0.60387ex
                  \kern -0.6em \intop}\nolimits_{#1}}}
\def\vintslides_#1{\mathchoice
          {\mathop{\kern 0.1em\vrule width 0.5em height 0.697ex depth -0.581ex
                  \kern -0.6em \intop}\nolimits_{\kern -0.4em#1}}
          {\mathop{\kern 0.1em\vrule width 0.3em height 0.697ex depth -0.604ex
                  \kern -0.4em \intop}\nolimits_{#1}}
          {\mathop{\kern 0.1em\vrule width 0.3em height 0.697ex depth -0.604ex
                  \kern -0.4em \intop}\nolimits_{#1}}
          {\mathop{\kern 0.1em\vrule width 0.3em height 0.697ex depth -0.604ex
                  \kern -0.4em \intop}\nolimits_{#1}}}

\newcommand{\aveint}[2]{\mathchoice
          {\mathop{\kern 0.2em\vrule width 0.6em height 0.69678ex depth -0.58065ex
                  \kern -0.8em \intop}\nolimits_{\kern -0.45em#1}^{#2}}
          {\mathop{\kern 0.1em\vrule width 0.5em height 0.69678ex depth -0.60387ex
                  \kern -0.6em \intop}\nolimits_{#1}^{#2}}
          {\mathop{\kern 0.1em\vrule width 0.5em height 0.69678ex depth -0.60387ex
                  \kern -0.6em \intop}\nolimits_{#1}^{#2}}
          {\mathop{\kern 0.1em\vrule width 0.5em height 0.69678ex depth -0.60387ex
                  \kern -0.6em \intop}\nolimits_{#1}^{#2}}}

\usepackage{tikz}
\usetikzlibrary{decorations.pathreplacing}
\usepackage{pgfplots}
\usepackage{tikz-3dplot}
\usetikzlibrary{arrows}

\numberwithin{equation}{section}

\newtheorem{theorem}{Theorem}[section]
\newtheorem{lemma}[theorem]{Lemma}
\newtheorem{proposition}[theorem]{Proposition}

\newtheorem*{main-thm}{Main Theorem}
\newtheorem*{main-lemma}{Main Lemma}

\theoremstyle{definition}

\newtheorem{remark}[theorem]{Remark}
\newtheorem{example}[theorem]{Example}

\newcommand{\abs}[1]{\left\vert#1\right\vert} 
\newcommand{\pare}[1]{\left(#1\right)} 
\newcommand{\braces}[1]{\left\{#1\right\}} 
\newcommand{\brackets}[1]{\left[#1\right]} 

\DeclareMathOperator*{\trace}{Tr}

\newcommand{\R}{\mathbb R}
\newcommand{\N}{\mathbb N}
\newcommand{\B}{\mathbb B}
\newcommand{\A}{\mathcal A}

\newcommand{\eps}{{\varepsilon}}

\renewcommand{\leq}{\leqslant}
\renewcommand{\geq}{\geqslant}

\newcommand{\sqmatrix}[4]{\begin{pmatrix}
				#1 & #2 \\
				#3 & #4
				\end{pmatrix}}

\def\Xint#1{\mathchoice
   {\XXint\displaystyle\textstyle{#1}}
   {\XXint\textstyle\scriptstyle{#1}}
   {\XXint\scriptstyle\scriptscriptstyle{#1}}
   {\XXint\scriptscriptstyle\scriptscriptstyle{#1}}
   \!\int}
\def\XXint#1#2#3{{\setbox0=\hbox{$#1{#2#3}{\int}$}
     \vcenter{\hbox{$#2#3$}}\kern-.5\wd0}}

\def\dashint{\Xint-}

\begin{document}


\title[Asymptotic H\"older Regularity for the Ellipsoid Process]{Asymptotic H\"older Regularity for the Ellipsoid Process}

\author[Arroyo]{\'Angel Arroyo}
\address{Department of Mathematics, University of Genoa, Via Dodecaneso 35, 16146 Genova, Italy}
\email{arroyo@dima.unige.it}

\author[Parviainen]{Mikko Parviainen}
\address{Department of Mathematics and Statistics, University of Jyv\"askyl\"a, PO~Box~35, FI-40014 Jyv\"askyl\"a, Finland}
\email{mikko.j.parviainen@jyu.fi}

\date{\today}

\keywords{Dynamic programming principle, local H\"older estimates, stochastic games, coupling of stochastic processes, ellipsoid process, equations in non-divergence form.} 
\subjclass[2010]{35B65, 35J15, 60H30, 60J10, 91A50}

\begin{abstract}
We obtain an asymptotic H\"older estimate for functions satisfying a dynamic programming principle arising from a so-called ellipsoid process. By the ellipsoid process we mean a generalization of the random walk where the next step in the process is taken inside a given space dependent ellipsoid. This stochastic process is related to elliptic equations in non-divergence form with bounded and measurable coefficients, and the regularity estimate is stable as the step size of the process converges to zero. 
The proof, which requires certain control on the distortion and the measure of the ellipsoids but not continuity assumption, is based on the coupling method.
\end{abstract}

\maketitle


\section{Introduction}

\subsection{Overview}

The Krylov-Safonov \cite{krylovs79} Hölder regularity result is one of the central results in the theory of non-divergence form elliptic partial differential equations with  bounded and measurable coefficients.  The result is not only important on its own right but also gives a flexible tool in the higher regularity and existence theory due to its rather  weak assumptions on the coefficients.

In this paper,  we study a quite general class of what we call ellipsoid processes. Ellipsoid processes are generalizations of the random walk where the next step in the process is taken inside a given space dependent ellipsoid $E_x$, and the value function $u_{\eps}$ satisfies the dynamic programming principle
\begin{align*}
u_\varepsilon(x)
				=
				\dashint_{E_x} u_\varepsilon(y)\ dy,
\end{align*}
as explained in  detail in the next section.
Their role among the discrete processes is somewhat similar  to the role of the linear uniformly elliptic partial differential equations with bounded and measurable coefficients among partial differential equations. Our main result, \Cref{MAIN-THM}, establishes an asymptotic Hölder regularity for a value function of an ellipsoid process under certain assumptions on the distortion and the measure of the ellipsoids but without any continuity assumption.

There is a classical well-known connection between the Brownian motion and the Laplace equation.
It was observed in the paper of Peres, Schramm, Sheffield and Wilson \cite{peresssw09} in discrete time that there is also a connection between the infinity Laplace equation and  a two-player random turn zero-sum game called tug-of-war. Similarly, a connection exists between the $p$-Laplacian, $1<p<\infty$ and different variants of tug-of-war game with noise \cite{peress08, manfredipr12}, as well as between the $1$-Laplacian and a deterministic two-player game \cite{kohns06}.    
Our study of ellipsoid processes is partly motivated by the aim of developing a more general approach that would be widely applicable to problems including the tug-of-war games with noise.  The mean value principle associated to the ellipsoids has also been discussed in the context of PDE theory for example by Pucci and Talenti  in \cite{puccit76}.

A more detailed overview of the proof is given in the next sections, but roughly speaking the proof is based on suitable couplings of probability measures related to the dynamic programming principle at different points and then look at the higher dimensional dynamic programming principle. The underlying idea is coming from the related stochastic processes: if we can show that with high probability the components of the coupled process coincide at some point, this will give a regularity estimate. In the usual random walk it is rather immediate that a good choice is a mirror point coupling of probability measures. However, since our ellipsoids can vary from point to point in a discontinuous fashion, finding good enough couplings becomes a nontrivial task.  At the end of the paper, we illustrate the main points of the proof by explicit examples, and also counterexamples demonstrating the role of the assumptions.

The coupling approach to the regularity of different variants of tug-of-war with noise was first developed in \cite{luirop18}  and applied in \cite{parviainenr16, arroyohp17,arroyolpr,han}.
As it turned out, for the continuous time diffusion processes and the Laplacian, the coupling method was utilized in connection to the regularity already at the beginning of 90's by Cranston \cite{cranston91}, utilizing the tools developed in \cite{lindvallr86}. For more recent works, see for example \cite{kusuoka15, kusuoka17}, which deal with linear PDEs under spatial continuity assumptions on the coefficients.  Actually, in continuous time, the lack of regularity can have some fundamental consequences: Nadirashvili showed in \cite{nadirashvili97} that there is not necessarily a unique diffusion, i.e.\ a unique solution to the martingale problem, nor is there necessarily a unique solution to the corresponding linear PDE with bounded and measurable coefficients. Some of the aspects in the method of coupling are similar to those of Ishii-Lions method \cite{ishiil90} as pointed out in \cite{porrettap13}, but the methods seem to have developed independently.

Let us also point out that the global approach to regularity used for example in  \cite{peress08, manfredipr12, lewickam17, lewicka} seems to be hard to adapt to our situation. This approach is based on coupling the same steps in different points so that the distance between the points is preserved, and continuing close to the boundary of the domain. Alternatively one can use the comparison with translated solutions.  In both cases the translation invariance is used and this is not available for the ellipsoid process.
 
This paper is organized as follows. In \Cref{sec:main-thm}, we make detailed statements. In \Cref{Sec. Preliminaries} we review some basic notions. \Cref{Sec. Comparison} presents the key lemmas in the coupling method. \Cref{Sec. >Nepsilon} is devoted to the proof of our main result in the case $|x-z|\gtrsim\varepsilon$, while \Cref{Sec. <Nepsilon} deals with the remaining case $|x-z|\lesssim\varepsilon$. Finally in \Cref{Sec. Examples} we give some examples showing that the assumptions are really needed in the method used here.

\subsection{Statement of the main result}
\label{sec:main-thm}
Given $0<\lambda\leq\Lambda<\infty$, let us denote by $\A(\lambda,\Lambda)$ the set of all symmetric $n\times n$ real matrices $A$ such that
\begin{equation*}
				\lambda|\xi|^2
				\leq
				\xi^\top\! A\,\xi
				\leq
				\Lambda|\xi|^2
\end{equation*}
for every $\xi\in\R^n$. We say that the matrices in $\A(\lambda,\Lambda)$ are \emph{uniformly elliptic}. We also refer to $\lambda$ and $\Lambda$ as the ellipticity constants of $A$.
\\

It turns out that each symmetric and positive definite matrix $A$ determines the shape and the orientation of an ellipsoid $E_A\subset\R^n$ centered at the origin and given by the formula
\begin{equation*}
				E_A
				:\,=
				\{y\in\R^n \,:\, y^\top\! A^{-1} y<1\}
				=
				A^{1/2}\B,
\end{equation*}
where $\B$ denotes the open unit ball in $\R^n$ and $A^{1/2}$ stands for the \emph{principal square root of $A$}. Moreover, the length of the principal semi-axes of $E_A$ is determined by the square root of the eigenvalues of $A$ (see \Cref{Sec. Preliminaries} for more details), and thus the distortion of $E_A$ coincides with the square root of the quotient between the largest and the smallest eigenvalue of $A$. In particular, for any $A\in\A(\lambda,\Lambda)$, the distortion of $E_A$ is bounded from above by $\sqrt{\Lambda/\lambda}$. Motivated by this, we say that the quotient $\Lambda/\lambda$ is the (maximum) \emph{distortion} of the matrices in $\A(\lambda,\Lambda)$.
\\

Let $\Omega\subset\R^n$ be a domain. We denote by $A:\Omega\to\A(\lambda,\Lambda)$ a matrix-valued function in $\Omega$ with measurable coefficients and values in $\A(\lambda,\Lambda)$. Given $\varepsilon>0$, for each $x\in\Omega$ we can define an ellipsoid $E_x$ centered at $x$ by
\begin{equation}\label{ellipsoid}
				E_x
				:\,=
				E_{\varepsilon,A,x}
				:\,=
				x+\varepsilon A(x)^{1/2}\,\B.
\end{equation}
We also assume that the ellipsoids have the same measure, in other words, $\det\{A(x)^{1/2}\}$ is constant for every $x\in\Omega$. This does not seriously affect the applicability of our result as discussed at the end of \Cref{Sec. Preliminaries}.
\\

For fixed $\varepsilon>0$, in this paper we deal with solutions $u_\eps$ of the \emph{dynamic programming principle} (\emph{DPP}) 
\begin{equation}\label{DPP}
				u_\varepsilon(x)
				=
				\dashint_{E_x} u_\varepsilon(y)\ dy
\end{equation}
for $x\in\Omega$. That is, $u_\varepsilon$ is a function whose value at a point $x\in\Omega$ coincides with its mean value over the ellipsoid $E_x$ (with respect a uniform probability distribution).
\\

Now we are in position of stating our main result, which asserts that, assuming a bound for the distortion of the ellipsoids, the solutions $u_\varepsilon$ to \eqref{DPP} are asymptotically H\"older continuous.

\begin{theorem}\label{MAIN-THM}
Let $\Omega\subset\R^n$ be a domain, $n\geq 2$ and $0<\lambda\leq\Lambda<\infty$ such that
\begin{equation}\label{distortion-bound}
				1
				\leq
				\frac{\Lambda}{\lambda}
				<
				\frac{n+1}{n-1}.
\end{equation}
Suppose  that $A:\Omega\to\A(\lambda,\Lambda)$ is a measurable mapping with constant determinant and $B_{2r}(x_0)\subset\Omega$ for some $r>0$. If $u_\varepsilon$ is a solution to \eqref{DPP}, then there exists some $\alpha=\alpha(n,\lambda,\Lambda)\in(0,1)$ such that
\begin{equation*}
				|u_\varepsilon(x)-u_\varepsilon(z)|
				\leq
				C(|x-z|^\alpha+\varepsilon^\alpha),
\end{equation*}
holds for every $x,z\in B_r(x_0)$ and some constant $C>0$ depending on  $n$, $\lambda$, $\Lambda$, $r$, $\alpha$ and $\sup_{B_{2r}}|u|$,  but independent of $\varepsilon$.
\end{theorem}

\subsection{The ellipsoid process and heuristic idea of the proof}\label{Sec. Ellipsoid Process}

The \emph{ellipsoid process} can be seen as a generalization of the random walk. Let us consider a sequence of points $\{x_0,x_1,x_2,\ldots\}\subset\Omega$ describing the location of a particle at each time $j=0,1,2,\ldots$ Steps of the particle are decided according to the following rule: suppose that the particle is placed at some point $x_j\in\Omega$, then the next  particle position $x_{j+1}$ is chosen randomly in the ellipsoid $E_{x_j}$ (according to a uniform probability distribution on $E_{x_j}$). This step is repeated until the particle exits $\Omega$ for the first time at some $x_\tau\notin\Omega$. Then the process stops and the amount $F(x_\tau)$ is collected, where $F$ is a pay-off function defined outside $\Omega$. The one step rule 
above defines a probability measure at every point. These probability measures induce a probability measure on the space of sequences according to the Kolmogorov construction.
 Using this probability measure, the expected pay-off of the ellipsoid process starting from $x_0\in\Omega$ is then
 \begin{equation*}
 				u_{\eps}(x_0)
				:\,=
				\mathbb{E}[F(x_\tau)\,|\,x_0].
 \end{equation*}
Moreover, the process is a Markov chain and by the Markov property the expected pay-off at a point $x$ coincides with the average of the expected pay-offs over the ellipsoid $E_x$, and thus it satisfies the DPP \eqref{DPP}.
\\

Given any solution $u_\varepsilon:\Omega\to\R$ to the original DPP \eqref{DPP}, we perform a change of variables to rewrite the DPP as
\begin{equation}\label{DPP2}
				u_\varepsilon(x)
				=
				\dashint_\B u_\varepsilon(x+\varepsilon A(x)^{1/2}\,y)\ dy.
\end{equation}
Then, we construct a $2n$-dimensional dynamic programming principle by defining  $G:\Omega\times\Omega\to\R$ as 
\begin{equation*}
				G(x,z)
				:\,=
				u_\varepsilon(x)-u_\varepsilon(z)
\end{equation*}
and we get that
\begin{equation}\label{2n-DPP}
				G(x,z)
				=
				\dashint_{\B} G(x+\varepsilon A(x)^{1/2}\, y,z+\varepsilon A(z)^{1/2}\, y)\ dy.
\end{equation}
Now, by means of the function $G$, the problem of the regularity of $u_\varepsilon$ becomes a question about the size of a solution $G$ to \eqref{2n-DPP} in $\Omega\times\Omega$. 

Observe that, due to the invariance of the unit ball $\B$ under orthogonal transformations, performing an orthogonal change of variables in \eqref{DPP2} for $u_\varepsilon(z)$, we see that the $2n$-dimensional DPP \eqref{2n-DPP} is equivalent to
\begin{equation}\label{2n-DPP-PQ}
				G(x,z)
				=
				\dashint_{\B} G(x+\varepsilon A(x)^{1/2} y,z+\varepsilon A(z)^{1/2}Q\, y)\ dy,
\end{equation}
for any orthogonal matrix $Q$. 
Note that the value of the integral with $G$ does not depend on $Q$. However, the choice of this matrix will become relevant in \Cref{Sec. Q}, where we estimate the right-hand side of \eqref{2n-DPP-PQ} for a certain explicit comparison function $f$ and show that it is a  supersolution for \eqref{2n-DPP-PQ} with a strict inequality. We will also show that the solution $G$ satisfies comparison principle with the supersolution $f$. By the explicit structure of $f$, this implies the desired regularity result.  

In stochastic terms, 
the above proof can be described by looking at a $2n$-dimensional stochastic process induced by the ellipsoid process started at $x_0$ and $z_0$, and by their probability measures coupled using the above coupling. The process continues until either one of the particles exits $\Omega$ for the first time (and then we impose the pay-off $2\sup_\Omega|u_\varepsilon|$), or both particles occupy the same position (and we impose the pay-off $0$). 
In this context, the key would be to show that the latter case occurs with a high enough probability.

Let us further illustrate the proof by taking as an example the particular case in which both $E_x$ and $E_z$ are unit balls centered at $x$ and $z$, respectively. In this case the ellipsoid process reduces to a random walk. Then, a suitable choice for the orthogonal matrix $Q$, i.e.\ the coupling, would be the matrix describing the reflection with respect to the $(n-1)$-dimensional hyperplane orthogonal to $x-z$ and passing through $(x+z)/2$.

However, in the general case where we have two different ellipsoids $E_x$ and $E_z$, it is not clear which couplings to choose.
Actually, for ellipsoids with large distortion it may happen that there is no \emph{good enough} coupling that lead us to obtain the sufficient estimates (see the examples in \Cref{Sec. Examples}).

\section{Preliminaries and notation}\label{Sec. Preliminaries}

In this work, $O(n)$ stands for the $n$-dimensional orthogonal group defined as
\begin{equation*}
				O(n)
				:\,=
				\{P\in\R^{n\times n}\,:\,P^\top P=PP^\top=I\},
\end{equation*}
where $P^\top$ stands for the transpose of $P$.

Note that, since $\lambda>0$, then every matrix $A$ in $\A(\lambda,\Lambda)$ is, in particular, positive definite.
This together with the symmetry of $A$ implies that there exists an orthogonal matrix $R\in O(n)$ and a diagonal matrix $D$ (with real and positive elements) such that
\begin{equation*}
				A=RDR^\top.
\end{equation*}
We denote by 
$D^{1/2}$ the square root of $D$, obtained by taking the square root of each element in the diagonal of $D$. In consequence, we define the \emph{principal square root} of a symmetric and positive semidefinite matrix $A$ as
\begin{equation}\label{A=RDR}
				A^{1/2}
				:\,=
				RD^{1/2}R^\top.
\end{equation}

In view of \eqref{A=RDR}, we see that the length of the principal semi-axes of the ellipsoid $E_A$ is given by the square root of the eigenvalues of $A$ (that is the diagonal entries of $D^{1/2}$), while the orientation of those semi-axes is given by the orthogonal matrix $R$.
Hence, by uniform ellipticity, the inclusions
\begin{equation}\label{inclusions}
				B_{\sqrt{\lambda}}\subset E_A \subset B_{\sqrt{\Lambda}}
\end{equation}
hold uniformly for every $A\in\A(\lambda,\Lambda)$. 
Moreover, the measure of the ellipsoid $E_A$ can be easily computed as $|E_A|=|\B| \sqrt{\det\{A\}}$.
\\

As we have defined in \eqref{ellipsoid}, let us fix $\varepsilon>0$ and $A:\Omega\to\A(\lambda,\Lambda)$, and let us explain the connection between the ellipsoids $E_x$, the DPP \eqref{DPP} and the
elliptic equation in non-divergence form
\begin{equation}\label{pde}
				\trace\{A(x) D^2u(x)\}
				=
				\sum_{i,j=1}^na_{ij}(x)\,u_{x_i x_j}(x)
				=
				0,
\end{equation}
where the coefficients $a_{ij}$ are the entries of $A(x)$.
We start by recalling the second order Taylor's expansion of a twice differentiable function $u$ at $x\in\Omega$,
\begin{equation*}
				u(x+y)
				=
				u(x)+\nabla u(x)^\top y+\frac{1}{2}\,y^\top\!D^2u(x)y+o(|y|^2),
\end{equation*}
for every small enough $y\in\R^n$. Here we have used the notation $v^\top w$ for the scalar product of two (column) vectors $v,w\in\R^n$.
Next, we compute the average with respect to $y$ over the ellipsoid $\varepsilon A(x)^{1/2}\B$ of the expansion above. Then the first order term vanishes by the symmetry of $\varepsilon A(x)^{1/2}\B$, that is
\begin{equation*}
				\dashint_{\varepsilon A(x)^{1/2}\B}\nabla u(x)^\top y\ dy
				=
				\nabla u(x)^\top\bigg(\dashint_{\varepsilon A(x)^{1/2}\B}y\ dy\bigg)
				=
				0,
\end{equation*}
while for the second order term we first rewrite it in trace form by recalling the equality $w^\top Mv=\trace\{M\, vw^\top\}$.
Then we need to compute the averaged integral over the ellipsoid of each entry in the matrix $yy^\top$,
\begin{equation*}
\begin{split}
				\dashint_{\varepsilon A(x)^{1/2}\B}yy^\top\ dy
				=
				~&
				\dashint_{\B}(\varepsilon A(x)^{1/2}\, y)(\varepsilon A(x)^{1/2}\, y)^\top\ dy
				\\
				=
				~&
				\varepsilon^2A(x)^{1/2}\bigg(\dashint_{\B}yy^\top\ dy\bigg)A(x)^{1/2}
				=
				\frac{\varepsilon^2}{n+2}\, A(x),
\end{split}
\end{equation*}
where in the first equality a change of variables has been performed and in the last equality we have used that the averaged integral in parenthesis is equal to the identity matrix divided by $n+2$.
Thus, the expansion for the average of $u$ over $E_x$ becomes
\begin{equation*}
				\dashint_{E_x}u(y)\ dy
				=
				u(x)+\frac{\varepsilon^2}{2(n+2)}\,\trace\{A(x) D^2u(x)\}+o(\varepsilon^2).
\end{equation*}
Therefore, we get the following asymptotic mean value property related to \eqref{pde}: let $u\in C^2(\Omega)$, then
\begin{equation*}
				\trace\{A(x) D^2u(x)\}
				=
				0
				\qquad
				\Longleftrightarrow
				\qquad
				u(x)
				=
				\dashint_{E_x}u(y)\ dy+o(\varepsilon^2).
\end{equation*}
However, as already pointed out, the solutions to 
\eqref{pde}
with bounded and measurable coefficients, with no further regularity assumptions, are not necessarily unique.
By a scaling argument we may assume that the determinant of $A$ is a constant function on $\Omega$, and thus the ellipsoids $E_x$ defined in \eqref{ellipsoid} have all the same measure.

\section{The coupling method for asymptotic regularity}\label{Sec. Comparison}

\subsection{The comparison function}

One of the keys in the proof of \Cref{MAIN-THM} is the construction of a suitable comparison function $f$ that is a supervalue for the function $G$ in the $2n$-dimensional DPP \eqref{2n-DPP}, and possesses  the desired regularity properties. 
We construct such a comparison function for the $2n$-dimensional DPP \eqref{2n-DPP} by defining explicitly the function $f_1:\R^{2n}\to\R$ by
\begin{equation}\label{f1}
				f_1(x,z)
				=
				C|x-z|^\alpha+\widetilde C|x+z|^2,
\end{equation}
where $\alpha\in(0,1)$ is a small enough exponent depending on $n$, $\lambda$ and $\Lambda$ (see \eqref{distortion-bound-alpha}),
\begin{equation*}
				\widetilde C
				=
				\frac{2}{3r^2}\,\sup_{B_{2r}}|u|,
\end{equation*}
and
\begin{equation}\label{C1}
				C>\frac{2}{r^\alpha}\,\sup_{B_{2r}}|u|
\end{equation}
is a constant depending on $n$, $\lambda$, $\Lambda$, $r$, $\alpha$ and $\sup_{B_{2r}}|u|$ (see \eqref{C2}, \eqref{C3} and \eqref{C4}). Here and in what follows, we use a shorthand $u:\,=u_\varepsilon$.
 Note that the first term in \eqref{f1} plays the role of the modulus of $\alpha$-H\"older continuity needed for obtaining the regularity result, while the second is just a small correction term introduced in order to ensure that
\begin{equation}\label{ineq-f1}
				|u(x)-u(z)|
				\leq
				f_1(x,z)
				\quad\quad \mbox{ when }\ x,z\in B_{2r}\setminus B_r.
\end{equation}
Indeed, if $x,z\in B_{2r}\setminus B_r$ such that $\abs{x-z}\leq r$, then
\begin{equation*}
				|x+z|^2
				=
				2\abs{x}^2+2\abs{z}^2-\abs{x-z}^2\geq 3r^2
\end{equation*}
and
\begin{equation*}
				|u(x)-u(z)|
				\leq
				2\sup_{B_{2r}}|u|
				=
				3\widetilde Cr^2\leq \widetilde C\abs{x+z}^2\leq f_1(x,z).
\end{equation*}
Otherwise, if $|x-z|>r$, using \eqref{C1}
we get
\begin{equation*}
				|u(x)-u(z)|
				\leq
				2\sup_{B_{2r}}|u|
				<
				Cr^\alpha
				<
				C|x-z|^\alpha
				<
				f_1(x,z).
\end{equation*}

However, due to the discrete nature of the DPP, the solutions $u=u_\varepsilon$ to \eqref{DPP} can present jumps in the small $\varepsilon$-scale. For that reason, we need to introduce an additional term in the comparison function in order to control the behavior of $u$ in this situation. This is, an \emph{annular step} function $f_2$ defined by
\begin{equation}\label{f2}
				f_2(x,z)
				=
				\left\{
				\begin{array}{cl}
				C^{2(2N-i)}\varepsilon^\alpha & \displaystyle \mbox{ if }\ \frac{i-1}{2}<\frac{|x-z|}{\sqrt\lambda\,\varepsilon}\leq \frac{i}{2}, \quad (i=0,1,\ldots,2N),
				\\[10pt]
				0 & \displaystyle \mbox{ if }\ \frac{|x-z|}{\sqrt\lambda\,\varepsilon}>N
				\end{array}
				\right.
\end{equation}
where $N\in\N$ is a large enough constant depending on $\lambda$, $\Lambda$ and $C$ that will be chosen later (see \eqref{N1}, \eqref{N2} and \Cref{Remark 4.2}). Note that $f_2=0$ whenever $|x-z|>N\sqrt\lambda\,\varepsilon$ and that $\sup f_2=C^{4N}\varepsilon^\alpha$ is reached on the set
\begin{equation*}
				S_0
				=
				\{(x,z)\in\R^{2n}\,:\,x=z\}.
\end{equation*}

\subsection{The counter assumption}

We choose 
\begin{equation*}
				f(x,z)
				\,:=
				f_1(x,z)-f_2(x,z).
\end{equation*}
as our comparison function. Since the terms in $f_1$ have been chosen so that \eqref{ineq-f1} holds and $\sup f_2=C^{4N}\varepsilon^\alpha$, then
\begin{equation}\label{ineq u - f}
				|u(x)-u(z)|
				\leq
				f(x,z)+C^{4N}\varepsilon^\alpha
				\quad\quad \mbox{ when }\ x,z\in B_{2r}\setminus B_r.
\end{equation}
In order to check that this estimate is also satisfied inside $B_r$, let us define a constant $K$ measuring the maximum difference between $|u(x)-u(z)|$ and $f(x,z)$ for $x$ and $z$ in $B_r$, i.e.
\begin{equation*}
				K
				:\,=
				\sup_{x,z\in B_r}(u(x)-u(z)-f(x,z)).
\end{equation*}

Our aim is to show that there exist suitable constants $C$ and $N$ so that the inequality \eqref{ineq u - f} also holds in $B_r$, i.e.
\begin{equation}\label{aim}
				|u(x)-u(z)|
				\leq
				f(x,z)+C^{4N}\varepsilon^\alpha
				\quad\quad \mbox{ when }\ x,z\in B_r.
\end{equation}
We proceed by contradiction. Assume that \eqref{aim} does not hold. Then, this implies that
\begin{equation}\label{contra}
				K
				>
				C^{4N}\varepsilon^\alpha. 
\end{equation}
As a direct consequence of the counter assumption, observe that from the definition of $K$ together with \eqref{ineq u - f} it holds that
\begin{equation}\label{ineq K B_2r}
				u(x)-u(z)
				\leq
				f(x,z)+K
				\qquad
				\mbox{ for every }
				x,z\in B_{2r}.
\end{equation}

\subsection{Statement of the key inequalities}

\label{sec:keys}

Let us assume that the counter assumption \eqref{contra} holds. It follows directly from the definition of $K$ that, for any $\eta>0$ to be fixed later, there exist $x_\eta,z_\eta\in B_r$ such that 
\begin{equation}\label{sup1}
				u(x_\eta)-u(z_\eta)-f(x_\eta,z_\eta)\geq K-\eta.
\end{equation}
In order to get a contradiction, we need to distinguish two different cases depending on the distance between the points $x_\eta$ and $z_\eta$ from \eqref{sup1}. The case $|x_\eta-z_\eta|<\frac{1}{2}\,\sqrt\lambda\,\varepsilon$ is simpler due to a cancellation effect that happens when the distance is small. This special case is presented in \Cref{Sec. short distance}. Now we focus on the case in which $|x_\eta-z_\eta|\geq\frac{1}{2}\,\sqrt\lambda\,\varepsilon$.

Our strategy is to utilize first the counter assumption together with \eqref{sup1} in order to obtain an inequality in terms of the comparison function $f$ (see \eqref{f < int f}). Then, by using the explicit form of the comparison function $f$, we choose adequately the constants $C$ and $N$ in such a way that $f$ satisfies the opposite strict inequality (\Cref{lemma f > int f}). This then is the desired contradiction giving the main result.

Starting from \eqref{sup1} and recalling the DPP \eqref{DPP}, we get
\begin{align}
\label{eq:key-contra}
K-\eta
				\leq
				~&
				u(x_\eta)-u(z_\eta)-f(x_\eta,z_\eta) \nonumber
				\\
				=
				~&
				\dashint_\B u(x_\eta+\varepsilon A(x_\eta)^{1/2}\,y)\ dy - \dashint_\B u(z_\eta+\varepsilon A(z_\eta)^{1/2}\,y)\ dy-f(x_\eta,z_\eta)
				\nonumber
				\\
				=
				~&
				\dashint_\B \Big[u(x_\eta+\varepsilon A(x_\eta)^{1/2}\,y) -  u(z_\eta+\varepsilon A(z_\eta)^{1/2}Q\,y)\Big]\ dy-f(x_\eta,z_\eta)
				\\
				\leq
				~&
				K+\dashint_\B f(x_\eta+\varepsilon A(x_\eta)^{1/2}\,y,z_\eta+\varepsilon A(z_\eta)^{1/2}Q\,y)\ dy-f(x_\eta,z_\eta), \nonumber
\end{align}
where in the second equality we have performed the orthogonal change of variables $y\mapsto Qy$ in the second integral, and in the last inequality we have recalled \eqref{ineq K B_2r}. We remark that the counter assumption has been applied here when using \eqref{ineq K B_2r}.

Thus for any $\eta>0$, let $x_\eta,z_\eta\in B_r$ that satisfy \eqref{sup1}, it follows that
\begin{equation}\label{f < int f}
				f(x_\eta,z_\eta)
				\leq
				\dashint_\B f(x_\eta+\varepsilon A(x_\eta)^{1/2}\,y,z_\eta+\varepsilon A(z_\eta)^{1/2}Q\,y)\ dy +\eta,
\end{equation}
where $Q\in O(n)$ is any fixed orthogonal matrix.

Now, the idea is to thrive for a contradiction by showing that the opposite strict inequality for \eqref{f < int f} holds for a certain choice of the coupling matrix $Q$.
As we will see in \Cref{Sec. Q}, this can only be done assuming a bound for the distortion of the matrices in $\A(\lambda,\Lambda)$.

\begin{lemma}\label{lemma f > int f}
Let $f=f_1-f_2$ be the comparison function, where $f_1$ and $f_2$ are the functions defined in \eqref{f1} and \eqref{f2}, respectively. Let $\eta=\eta(\varepsilon)=\varepsilon^2>0$ and assume that  $\Lambda/\lambda$ is bounded as in \eqref{distortion-bound}. Let $x,z\in B_r$ such that $|x-z|>\frac{1}{2}\sqrt\lambda\,\varepsilon$, then there exists an orthogonal matrix $Q\in O(n)$ such that
\begin{equation}\label{f > int f}
				f(x,z)
				>
				\dashint_\B f(x+\varepsilon A(x)^{1/2}\,y,z+\varepsilon A(z)^{1/2}Q\,y)\ dy +\eta
\end{equation}
holds.
\end{lemma}

Observe that, choosing $x$ and $z$ so that $x=x_\eta$ and $z=z_\eta$ in \Cref{lemma f > int f}, we get a contradiction between \eqref{f < int f} and \eqref{f > int f}, and this implies the falseness of \eqref{contra}, so \eqref{aim} holds. This, together with the short distance case presented in \Cref{Sec. short distance}, yields the desired asymptotic H\"older regularity estimate and the proof of \Cref{MAIN-THM} is complete.

We can assume without loss of generality that both $x_\eta$ and $z_\eta$ from \eqref{sup1} lie in the first coordinate axis of $\R^n$, i.e.\ $x_\eta,z_\eta\in\mathrm{span}\{e_1\}$. Furthermore, from now on we will assume that
\begin{equation*}
				\frac{x_\eta-z_\eta}{|x_\eta-z_\eta|}
				=
				e_1,
\end{equation*}
where $e_1=(1,0,\ldots,0)^\top$.

We split the proof of \Cref{lemma f > int f} into two cases depending on the size of $|x-z|$ in comparison with the constant $N\sqrt\lambda\,\varepsilon$, using different arguments and distinguishing between the case in which $f_2= 0$ (\Cref{Sec. >Nepsilon}) and the case in which $f_2\neq0$ (\Cref{Sec. <Nepsilon}).

\section{Case $|x_\eta-z_\eta|>N\sqrt\lambda\,\varepsilon$}\label{Sec. >Nepsilon}

In this section, we prove  \Cref{lemma f > int f} in the case $|x-z|>N\sqrt\lambda\,\varepsilon$, and thus, as pointed out in \Cref{sec:keys}, complete the proof of the main result \Cref{MAIN-THM} in this particular case. 

Since $|x-z|>N\sqrt\lambda\,\varepsilon$, by \eqref{f2} we have that $f_2(x,z)=0$. Thus
\begin{equation}\label{f}
				f(x,z)
				=f_1(x,z)
				=
				C|x-z|^\alpha+\widetilde C|x+z|^2.
\end{equation}
Our aim is to show that the inequality \eqref{f > int f} in \Cref{lemma f > int f} holds.
The idea is to use the explicit form of the comparison function in order to estimate the right-hand side of \eqref{f < int f}, and then check that, for a convenient choice of the constants, the estimate is strictly bounded from above by $f(x,z)$.

\subsection{Taylor's expansion for $f$}

In this section we compute the Taylor's expansion of $f_1$ in order to obtain the desired estimates for the comparison function. This is similar to \cite{luirop18,arroyohp17} and \cite{arroyolpr}. However, for the convenience of the reader and expository reasons we write down the details.

\begin{lemma}
Let $f$ be the comparison function \eqref{f} and suppose that $x,z\in B_r$ satisfy $(x-z)/|x-z|=e_1$. Then the inequality
\begin{multline}\label{f-taylor}
				f(x+h_x,z+h_z)-f(x,z)
				\\
\begin{split}
				\leq
				~&
				C\,\alpha|x-z|^{\alpha-1}e_1^\top(h_x-h_z)+2\widetilde C(x+z)^\top(h_x+h_z)
				\\
				~&
				+ \frac{C}{2}\,\alpha|x-z|^{\alpha-2}\trace\left\{\sqmatrix{\alpha-1}{0}{0}{I} (h_x-h_z)(h_x-h_z)^\top\right\}
				\\
				~&
				+ (16\widetilde C\Lambda r^{2-\alpha}+1)\abs{x-z}^{\alpha-2}\,\varepsilon^2
\end{split}
\end{multline}
holds for every $|h_x|,|h_z|\leq\sqrt{\Lambda}\, \varepsilon$.
\end{lemma}

\begin{proof}
First observe that,
\begin{multline}\label{taylor-M-part}
				|(x+h_x)+(z+h_z)|^2-|x+z|^2
				\\
\begin{split}
				=
				~&
				2(x+z)^\top(h_x+h_z)+|h_x+h_z|^2
				\\
				\leq
				~&
				2(x+z)^\top(h_x+h_z)+4\Lambda\,\varepsilon^2
				\\
				\leq
				~&
				2(x+z)^\top(h_x+h_z)+16\Lambda r^{2-\alpha}|x-z|^{\alpha-2}\varepsilon^2
\end{split}
\end{multline}
for every $|h_x|,|h_z|<\sqrt\Lambda\,\varepsilon$, where in the last inequality we have used that $|x-z|\leq 2r$ and $\alpha-2<0$.
Next, let us recall the second order Taylor's expansion of a $2n$-dimensional function $\phi(x,z)$,
\begin{multline*}
				\phi(x+h_x,z+h_z)-\phi(x,z)
				\\
				=
				D\phi(x,z)^\top
				\begin{pmatrix}
				h_x \\
				h_z
				\end{pmatrix}
				+\frac{1}{2}\begin{pmatrix}
				h_x^\top \ \
				h_z^\top
				\end{pmatrix}
				D^2\phi(x,z)
				\begin{pmatrix}
				h_x \\
				h_z
				\end{pmatrix}
				+\mathcal{E}_{x,z}(h_x,h_z),
\end{multline*}
where $\mathcal{E}_{x,z}(h_x,h_z)$ is the error term in the Taylor's expansion.
Developing the terms in the expansion, we equivalently have that
\begin{multline}\label{2n-taylor}
				\phi(x+h_x,z+h_z)-\phi(x,z)
				\\
				=
				D_x\phi(x,z)^\top h_x +  D_z\phi(x,z)^\top h_z
				\\
				+\frac{1}{2}\trace\left\{
				 D_{xx}\phi(x,z) h_xh_x^\top +  D_{zz}\phi(x,z) h_zh_z^\top + 2  D_{xz}\phi(x,z)h_zh_x^\top
				\right\}
				\\
				+\mathcal{E}_{x,z}(h_x,h_z).
\end{multline}
We use the formulas
\begin{equation*}
				D_x|x|
				=
				\frac{x}{|x|}
				\qquad \mbox{ and } \qquad
				D_{xx}|x|
				=
				\frac{1}{|x|}\bigg(I-\frac{xx^\top}{|x|^2}\bigg),
\end{equation*}
which, since $(x-z)/|x-z|=e_1$ by assumption, give us that
\begin{equation*}
				D_x|x-z|
				=
				e_1
				\quad \mbox{ and } \quad
				D_{xx}|x-z|
				=
				-D_{xz}|x-z|
				=
				\frac{1}{|x-z|}(I-e_1e_1^\top).
\end{equation*}
Thus, differentiating $|x-z|^\alpha$ with respect to $x$ and $z$ we get
\begin{equation}\label{D1}
				D_x|x-z|^\alpha=\alpha|x-z|^{\alpha-1}e_1,
				\qquad
				D_z|x-z|^\alpha=-\alpha|x-z|^{\alpha-1}e_1
\end{equation}
and
\begin{equation}\label{D2}
\begin{split}
				D_{xx}|x-z|^\alpha
				=
				~&
				D_{zz}|x-z|^\alpha
				=
				-D_{xz}|x-z|^\alpha
				\\
				=
				~&
				\alpha|x-z|^{\alpha-2}\big[-(1-\alpha)e_1e_1^\top+(I-e_1e_1^\top)\big]
				\\
				=
				~&
				\alpha|x-z|^{\alpha-2}\sqmatrix{\alpha-1}{0}{0}{I_{n-1}}.
\end{split}
\end{equation}
Next we estimate the error term. Since
\begin{equation*}
				\frac{\partial^3}{\partial t^3}\, t^\alpha
				=
				\alpha(1-\alpha)(2-\alpha)t^{\alpha-3},
\end{equation*}
by Taylor's theorem, whenever $|x-z|>2\sqrt\Lambda\,\varepsilon$, then
\begin{equation*}
				\mathcal{E}_{x,z}(h_x,h_z)
				\leq
				\frac{\alpha(1-\alpha)(2-\alpha)}{6}
				\abs{\begin{pmatrix}
				h_x \\
				h_z
				\end{pmatrix}}^3
				(|x-z|-2\sqrt\Lambda\,\varepsilon)^{\alpha-3},
\end{equation*}
holds for every $\abs{h_x},\abs{h_z}\leq\sqrt{\Lambda}\,\varepsilon$. Fix
\begin{equation}\label{N1}
				N
				\geq
				4\sqrt\frac\Lambda\lambda,
\end{equation}
then by hypothesis $|x-z|>N\sqrt\lambda\,\varepsilon\geq 4\sqrt\Lambda\,\varepsilon$, and since $0<\alpha<1$, we can estimate
\begin{equation*}
				(|x-z|-2\sqrt\Lambda\,\varepsilon)^{\alpha-3}
				\leq
				\pare{\frac{\abs{x-z}}{2}}^{\alpha-3}.
\end{equation*}
Inserting this in the estimate for the error term together with $|h_x|,|h_z|\leq\sqrt\Lambda\,\varepsilon$ we obtain
\begin{equation}\label{error-term}
\begin{split}
				\mathcal{E}_{x,z}(h_x,h_z)
				&
				\leq
				\frac{1}{3}(2\Lambda\varepsilon^2)^{3/2}\pare{\frac{\abs{x-z}}{2}}^{\alpha-3}
				\\
				&
				=
				\frac{2^{9/2-\alpha}}{3}\,\Lambda^{3/2}\,\frac{\varepsilon}{|x-z|}\,|x-z|^{\alpha-2}\,\varepsilon^2
				\\
				&
				<
				2^{7/2}\,\Lambda^{3/2}\,\frac{\varepsilon}{|x-z|}\,|x-z|^{\alpha-2}\,\varepsilon^2
				\\
				&
				<
				\frac{2^{7/2}\Lambda}{N}\sqrt{\frac{\Lambda}{\lambda}}\,|x-z|^{\alpha-2}\,\varepsilon^2,
\end{split}
\end{equation}
where in the last inequality we have used $|x-z|>N\sqrt\lambda\,\varepsilon$ again.
Plugging \eqref{D1}, \eqref{D2} and \eqref{error-term} into the terms in \eqref{2n-taylor} with $\phi(x,z)=|x-z|^\alpha$ we obtain
\begin{multline*}
				|(x+h_x)-(z+h_z)|^\alpha-|x-z|^\alpha
				\\
\begin{split}
				=
				~&
				\alpha|x-z|^{\alpha-1}e_1^\top(h_x-h_z) 
				\\
				~&
				+
				\frac{1}{2}\,\alpha|x-z|^{\alpha-2}\trace\bigg\{\sqmatrix{\alpha-1}{0}{0}{I}(h_x-h_z)(h_x-h_z)^\top\bigg\}
				\\
				~&
				+\frac{2^{7/2}\Lambda}{N}\sqrt{\frac{\Lambda}{\lambda}}\,|x-z|^{\alpha-2},
\end{split}
\end{multline*}
Thus combining this and \eqref{taylor-M-part} as in \eqref{f} and choosing a large enough natural number $N\in\N$ such that
\begin{equation}\label{N2}
				N>2^{7/2}\Lambda\sqrt{\frac{\Lambda}{\lambda}}\,C,
\end{equation}
we get \eqref{f-taylor}.
\end{proof}

\begin{remark}\label{Remark 4.2}
We could choose $N\in\N$ taking into account the fact that, as we will see later in \Cref{key-lemma-1}, the distortion needs to be bounded by certain constant depending on $n$ and $\alpha$ which is less than $3$. Then, it is enough to choose $N\geq 4\sqrt3$ and $N\geq \sqrt{2^{7}\cdot 3}\,\Lambda C$ instead of \eqref{N1} and \eqref{N2}, respectively.
\end{remark}

The desired estimate then follows after averaging the inequality \eqref{f-taylor} from the previous Lemma.

\begin{lemma}\label{expansion-lemma}
Let $f$ be the comparison function \eqref{f} and suppose that $x,z\in B_r$ satisfy $(x-z)/|x-z|=e_1$. Then the inequality
\begin{multline}\label{integral-taylor}
				\dashint_\B f\big(x+\varepsilon A(x)^{1/2}\,y,z+\varepsilon A(z)^{1/2}Q\,y\big)\ dy-f(x,z)
				\\
				\leq
				|x-z|^{\alpha-2}\bigg[
				\frac{C\alpha}{2}\,
				T_\alpha(A(x),A(z),Q)
				+16\widetilde C\Lambda r^{2-\alpha}+1\,\bigg]\,\varepsilon^2
\end{multline}
holds for any fixed orthogonal matrix $Q\in O(n)$, where
\begin{multline*}
				T_\alpha(A_1,A_2,Q)
				\\
				:\,=
				\dashint_\B\trace\left\{\sqmatrix{\alpha-1}{0}{0}{I} \Big(\big(A_1^{1/2}-A_2^{1/2}Q\big)y\Big)\Big(\big(A_1^{1/2}-A_2^{1/2}Q\big)y\Big)^{\top}\right\}\ dy.
\end{multline*}
\end{lemma}

\begin{proof}
Replace $h_x=\varepsilon A(x)^{1/2}\,y$ and $h_z=\varepsilon A(z)^{1/2}Q\,y$ in \eqref{f-taylor} and note that, by symmetry, averaging over $\B$, the first order terms vanish and we get \eqref{integral-taylor}.
\end{proof}

\subsection{The optimal orthogonal coupling}\label{Sec. Q}

In view of \Cref{expansion-lemma}, we need to ensure first the negativity of
the right-hand side of \eqref{integral-taylor}
in order to show that \eqref{f > int f} holds. For that reason, in this subsection, our aim is to check that for every pair of matrices $A_1,A_2\in\A(\lambda,\Lambda)$ there exists an orthogonal matrix $Q\in O(n)$ (depending on $A_1$, $A_2$ and $\alpha$) so that
\begin{equation}\label{ineq-Q}
				\dashint_\B\trace\left\{\sqmatrix{\alpha-1}{0}{0}{I} \Big(\big(A_1^{1/2}-A_2^{1/2}Q\big)y\Big)\Big(\big(A_1^{1/2}-A_2^{1/2}Q\big)y\Big)^{\top}\right\}\ dy
				<
				0
\end{equation}
holds.

This is equivalent to the question of which is the best coupling matrix in \eqref{2n-DPP-PQ}. Heuristically, since $\alpha-1<0$, the idea is to choose the orthogonal matrix $Q$ such that the difference (in average) between $A_1^{1/2}y$ and $A_2^{1/2}Qy$ projected over the first component is much larger than projected over the orthogonal subspace $\{e_1\}^\bot$ (see \Cref{fig. mirror point}).
\begin{figure}[h]
\begin{tikzpicture}[scale=1]
\draw (0,-1.5)--(0,1.5);
\draw (-3.5,0)--(3.5,0);
\filldraw (-2,0) circle (1pt);
\filldraw (2,0) circle (1pt);
\draw [thick, rotate around={30:(-2,0)}] (-2,0) ellipse (1.1547 and 0.7071);
\draw [thick, rotate around={-45:(2,0)}] (2,0) ellipse (1 and 0.81649);
\draw (-2,1.4142) node {{\small{$x+\varepsilon A(x)^{1/2}y$}}};
\draw (2,1.4142) node {{\small{$z+\varepsilon A(z)^{1/2}Qy$}}};
\draw (-2-1.4142,-1.4142) node [above] {$E_x$};
\draw (2+1.4142,-1.4142) node [above] {$E_z$};
\draw [->,thick, rotate around={0:(-2,0)}] (-2,0) -- (-2+1,1/3);
\draw [->,thick, rotate around={0:(2,0)}] (2,0) -- (2-0.85537,0.26560);
\end{tikzpicture}
\caption{Illustration of the coupling aiming at negativity of the left-hand side in \eqref{ineq-Q}.}
\label{fig. mirror point}
\end{figure}
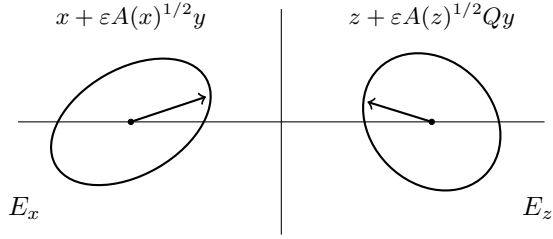

We start by rewriting \eqref{ineq-Q} in a more suitable form. By the linearity of the trace, we can integrate term by term inside the trace so we get
\begin{multline}\label{trace-formula}
				\dashint_\B\trace\left\{\sqmatrix{\alpha-1}{0}{0}{I} \Big(\big(A_1^{1/2}-A_2^{1/2}Q\big)y\Big)\Big(\big(A_1^{1/2}-A_2^{1/2}Q\big)y\Big)^\top\right\}\ dy
				\\
\begin{split}
				=
				~&
				\trace\left\{\sqmatrix{\alpha-1}{0}{0}{I}(A_1^{1/2}-A_2^{1/2}Q)\bigg(\dashint_\B\,yy^\top\ dy\bigg)(A_1^{1/2}-A_2^{1/2}Q)^\top\right\}
				\\
				=
				~&
				\frac{1}{n+2}\,
				\trace\left\{\sqmatrix{\alpha-1}{0}{0}{I}(A_1^{1/2}-A_2^{1/2}Q)(A_1^{1/2}-A_2^{1/2}Q)^\top\right\}
				\\
				=
				~&
				\frac{1}{n+2}\,
				\trace\left\{\sqmatrix{\alpha-1}{0}{0}{I}(A_1+A_2-2A_2^{1/2}QA_1^{1/2})\right\},
\end{split}
\end{multline}
where in the second equality we have used that
\begin{equation*}
				\dashint_\B yy^\top\ dy = \frac{1}{n+2}\, I.
\end{equation*}
Then, \eqref{ineq-Q} holds if and only if
\begin{equation*}
				\trace\left\{\sqmatrix{\alpha-1}{0}{0}{I}(A_1+A_2-2A_2^{1/2}QA_1^{1/2})\right\}
				<
				0.
\end{equation*} 
In order to accomplish this, our strategy is to find a matrix in $O(n)$ minimizing the map
\begin{equation}\label{Q-map}
				Q\longmapsto\trace\left\{\sqmatrix{\alpha-1}{0}{0}{I}(A_1+A_2-2A_2^{1/2}QA_1^{1/2})\right\},
\end{equation}
where $Q$ ranges the orthogonal group $O(n)$.
For that reason, first we remark that the above map is continuous and defined on the compact space $O(n)$. Thus, there exists an orthogonal matrix, say $Q_0$, depending on $A_1$ and $A_2$ such that
\begin{equation*}
				\trace\left\{\sqmatrix{\alpha-1}{0}{0}{I}A_2^{1/2}Q_0A_1^{1/2}\right\}
				\geq
				\trace\left\{\sqmatrix{\alpha-1}{0}{0}{I}A_2^{1/2}QA_1^{1/2}\right\}
\end{equation*}
for every $Q\in O(n)$. The matrix $Q_0$ represents the best possible coupling in \eqref{ineq-Q}. Then, the following step will be to impose the negativity of the resulting term as a sufficient condition for \eqref{f > int f}. As a consequence of this, it will be made clear the need of imposing a bound for the distortion of the matrices $A_1$ and $A_2$.
\\

We start by showing the following lemma, which is stated for any square matrix $M$ and it is a result of optimization over the orthogonal group by using the singular value decomposition of matrices.

\begin{lemma}\label{lemma-max-trace}
Let $n\geq 2$ and $M\in\R^{n\times n}$. Then $M^\top\!M$ is symmetric and positive semidefinite and satisfies
\begin{equation}\label{max-trace}
				\max_{Q\in O(n)}\trace\{MQ\}
				=
				\trace\{(M^\top\! M)^{1/2}\}
				\geq
				n\left|\det\{M\}\right|^{1/n}.
\end{equation}
\end{lemma}

\begin{proof} 
By the singular value decomposition, any square matrix $M$ can be factorized as $M=R_1S^{1/2}R_2^\top$, where $R_1,R_2\in O(n)$ and $S$ is a diagonal matrix containing the eigenvalues of $M^\top\! M$, which are real and non-negative since $M^\top\! M$ is symmetric and a positive semidefinite. Then, for any $Q\in O(n)$,
\begin{equation*}
				\trace\{MQ\}
				=
				\trace\{R_1S^{1/2}R_2^\top Q\}
				=
				\trace\{S^{1/2}R_2^\top QR_1\},
\end{equation*}
where $\trace\{AB\}=\trace\{BA\}$ has been used in the second equality. Since $O(n)$ together with the usual matrix multiplication has group structure, then we can select $\widetilde Q=R_2^\top QR_1\in O(n)$ and thus 
\begin{equation*}
				\max_{Q\in O(n)}\trace\{MQ\}
				=
				\max_{\widetilde Q\in O(n)}\trace\{S^{1/2}\widetilde Q\}.
\end{equation*}
Observe that, since $S^{1/2}$ is a diagonal matrix with non-negative entries, the maximum in the right-hand side of the previous equations is attained for $\widetilde Q=I$, i.e.
\begin{equation*}
				\max_{Q\in O(n)}\trace\{S^{1/2}Q\}
				=
				\trace\{S^{1/2}\}
				=
				\trace\{(M^\top\! M)^{1/2}\}.
\end{equation*}
Therefore, we derive the equality in \eqref{max-trace}.
In order to show the inequality in \eqref{max-trace}, let us write  $\widetilde M=(M^\top\! M)^{1/2}$ and recall the following inequality
\begin{equation*}
				\trace\{\widetilde M\}
				\geq
				n(\det\{\widetilde M\})^{1/n},
\end{equation*}
which follows from the inequality of arithmetic and geometric means and the fact that all the eigenvalues of $\widetilde M$ are real and non-negative. Then the proof is completed by observing that $\det\{\widetilde M\}=\left|\det\{M\}\right|$.
\end{proof}

\begin{remark}\label{optimal-orthogonal-coupling}
Out of curiosity, observe that if $Q\in O(n)$ and $M\in\R^{n\times n}$, the distance between $Q$ and $M^\top$ is then given by
\begin{equation*}
\begin{split}
				\|M^\top-Q\|^2
				=
				~&
				\trace\{(M^\top-Q)^\top(M^\top-Q)\}
				\\
				=
				~&
				\trace\{MM^\top+I-2MQ\}
				\\
				=
				~&
				\|M\|^2+n-2\trace\{MQ\}.
\end{split}
\end{equation*}
Therefore, the problem of finding the nearest orthogonal matrix
to a given matrix $M^\top$ is equivalent to the problem of maximizing $\trace\{MQ\}$ among all matrices $Q$ in $O(n)$.
Moreover, assuming that $\det\{M\}\neq 0$, the solution to this problem is attained at the orthogonal matrix $Q_0=(M^\top\! M)^{-1/2}M^\top$, i.e.\ $Q_0$ is the nearest orthogonal matrix to $M^\top$ and it
satisfies
\begin{equation*}
				\trace\{MQ\}
				\leq
				\trace\{MQ_0\}
				=
				\trace\{(M^\top\! M)^{1/2}\}
\end{equation*}
for every $Q\in O(n)$.
\end{remark}

Next we apply the previous lemma to obtain the following estimate.

\begin{lemma}\label{lemma-min-trace}
Let $0<\lambda\leq\Lambda<\infty$ and $\alpha\in(0,1)$. Then
\begin{multline}\label{min-trace}
				\min_{Q\in O(n)}\trace\left\{\sqmatrix{\alpha-1}{0}{0}{I}(A_1+A_2-2A_2^{1/2}QA_1^{1/2})\right\}
				\\
				\leq
				2\Big[(n-1)\Lambda-\big((1-\alpha)+n(1-\alpha)^{1/n}\big)\lambda\Big].
\end{multline}
holds for every pair of matrices $A_1,A_2\in\A(\lambda,\Lambda)$.
\end{lemma}

\begin{proof}
First notice that the first two terms in the trace of \eqref{min-trace} do not depend on $Q$, so we can bound them directly by recalling the uniform ellipticity of the matrices in $\A(\lambda,\Lambda)$, i.e.,
\begin{equation*}
\begin{split}
				\trace\left\{\sqmatrix{\alpha-1}{0}{0}{I}(A_1+A_2)\right\}
				\leq
				~&
				2\max_{A\in\A(\lambda,\Lambda)}\trace\left\{\sqmatrix{\alpha-1}{0}{0}{I}A\right\}
				\\
				=
				~&
				2\Big[(n-1)\Lambda-(1-\alpha)\lambda\Big].
\end{split}
\end{equation*}
for every pair of matrices $A_1,A_2\in\A(\lambda,\Lambda)$.
For the remaining term, we recall \Cref{lemma-max-trace} with
\begin{equation*}
				M
				=
				A_1^{1/2}\sqmatrix{\alpha-1}{0}{0}{I}A_2^{1/2}.
\end{equation*}
Then,
\begin{equation*}
\begin{split}
				\max_{Q\in O(n)}\trace\left\{\sqmatrix{\alpha-1}{0}{0}{I}A_2^{1/2}QA_1^{1/2}\right\}
				\geq
				~&
				n\abs{\det\braces{A_1^{1/2}\sqmatrix{\alpha-1}{0}{0}{I}A_2^{1/2}}}^{1/n}
				\\
				=
				~&
				n(1-\alpha)^{1/n}\abs{\det\{A_1\}\det\{A_2\}}^{\frac{1}{2n}}
				\\
				\geq
				~&
				n(1-\alpha)^{1/n}\lambda,
\end{split}
\end{equation*}
where in the second inequality we have used the uniform ellipticity to get that the derminant of each matrix is bounded from below by $\lambda^n$. Then \eqref{min-trace} follows.
\end{proof}

The following is the main result of this section and provides a sufficient condition for \eqref{f > int f} whenever $|x-z|>N\sqrt\lambda\,\varepsilon$, concluding the proof of \Cref{lemma f > int f} in this case.

\begin{proposition}\label{key-lemma-1}
Let $n\geq 2$, $0<\lambda\leq\Lambda<\infty$ and $\alpha\in(0,1)$ such that
\begin{equation}\label{distortion-bound-alpha}
				1
				\leq
				\frac{\Lambda}{\lambda}
				<
				\frac{n(1-\alpha)^{1/n}+(1-\alpha)}{n-1}.
\end{equation}
For $x,z\in B_r$ such that $|x-z|>N\sqrt\lambda\,\varepsilon$, there exists a coupling matrix $Q_0\in O(n)$ and a large enough constant $C>0$ such that
\begin{equation*}
				\dashint_\B f\big(x+\varepsilon A(x)^{1/2}\,y,z+\varepsilon A(z)^{1/2}Q_0\,y\big)\ dy
				<
				f(x,z)-\varepsilon^2,
\end{equation*}
where $f(x,z)=f_1(x,z)=C|x-z|^\alpha+\widetilde C|x+z|^2$.
\end{proposition}

\begin{proof}
Select $Q_0\in O(n)$ minimizing \eqref{Q-map} with $A_1=A(x)$, $A_2=A(z)$. Then, recalling \Cref{expansion-lemma} together with \eqref{trace-formula} and \Cref{lemma-min-trace}, it turns out that
\begin{multline*}
				\dashint_\B f\big(x+\varepsilon A(x)^{1/2}\,y,z+\varepsilon A(z)^{1/2}Q_0\,y\big)\ dy-f(x,z)
				\\
				\leq
				|x-z|^{\alpha-2}\pare{
				-\tau C
				+16\widetilde C\Lambda r^{2-\alpha}+1 }\,\varepsilon^2,
\end{multline*}
where
\begin{equation*}
				\tau=\tau(n,\lambda,\Lambda)
				:\,=
				-\frac{\alpha}{n+2}\Big[(n-1)\Lambda-\big((1-\alpha)+n(1-\alpha)^{1/n}\big)\lambda\Big]
				>
				0
\end{equation*}
by hypothesis. Finally, since $|x-z|<2r$, choosing large enough
\begin{equation}\label{C2}
				C
				>
				\frac{1}{\tau}\,(16\widetilde C\Lambda r^{2-\alpha}+4 r^{2-\alpha}+1)
\end{equation}
we ensure that the right-hand side in the previous inequality is less than $-\varepsilon^2$ and the proof is finished.
\end{proof}

\begin{remark}
Note that, as a consequence of \Cref{key-lemma-1}, the condition \eqref{distortion-bound} in the statement of \Cref{MAIN-THM} can be replaced by \eqref{distortion-bound-alpha}. Indeed, since the right-hand side of \eqref{distortion-bound-alpha} is decreasing on $\alpha$ and
\begin{equation*}
				\frac{n(1-\alpha)^{1/n}+(1-\alpha)}{n-1}
				\xrightarrow[\ \alpha\to 0\ ]{}
				\frac{n+1}{n-1},
\end{equation*}
it turns out that the condition \eqref{distortion-bound-alpha} is weaker than \eqref{distortion-bound}. 
Therefore, this improves the bound for the distortion in \Cref{MAIN-THM} and it provides an explicit dependence between the dimension, the ellipticity constants and the H\"older exponent of the solutions.
\end{remark}

\begin{remark}
It is worth noting that in the case of continuous coefficients (that is, when $A:\Omega\to\A(\lambda,\Lambda)$ is continuous) the bound for the distortion \eqref{distortion-bound} can be dropped from the assumptions in \Cref{MAIN-THM}. This is due to the fact that under the continuity assumption the difference between $A_1=A(x)$ and $A_2=A(z)$ is small for every $x$ and $z$ close enough. In other words, we can assume that $A_1$ and $A_2$ are almost the same matrix. Then the negativity in \eqref{ineq-Q} can be deduced from the constant coefficient case $A_1=A_2$, for which \eqref{ineq-Q} holds independently of the ellipticity constants.

To see this, let us first recall that one of the key steps in the proof of the asymptotic H\"older estimate is the minimization of \eqref{Q-map} among all orthogonal matrices $Q\in O(n)$, and then to deduce the conditions under which such minimum is negative. This condition is expressed through the formula \eqref{distortion-bound-alpha} and is only relevant in the large distance case. In fact, in the medium and short distance cases we assume that $\Lambda/\lambda<3$ just for convenience, but we could adapt the proof without this assumption.
In order to replace the assumption \eqref{distortion-bound} by the continuity of $A(\cdot)$, we first focus on the constant coefficients case, that is, when both $A_1$ and $A_2$ in \eqref{Q-map} are equal to a given fixed matrix $A\in\A(\lambda,\Lambda)$. Then it is possible to show the following estimate
\begin{equation}\label{ctt A tr<0}
	\min_{Q\in O(n)}\trace\left\{\sqmatrix{\alpha-1}{0}{0}{I}(2A-2A^{1/2}QA^{1/2})\right\}
	\leq
	-4(1-\alpha)\lambda,
\end{equation}
which holds for every $A\in\A(\lambda,\Lambda)$ and $\alpha\in(0,1)$. As a direct consequence of \eqref{ctt A tr<0}, it turns out that \eqref{ineq-Q} holds in the case $A_1=A_2$ independently of $0<\lambda\leq\Lambda<\infty$ and $\alpha\in(0,1)$. Moreover, by virtue of \Cref{optimal-orthogonal-coupling}, this minimum is attained for the nearest orthogonal matrix to \begin{equation*}
	A^{1/2}\sqmatrix{\alpha-1}{0}{0}{I}A^{1/2}.
\end{equation*}

Now, returning to the general case, if $A_1,A_2\in\A(\lambda,\Lambda)$, we can rewrite \eqref{Q-map} as follows,
\begin{multline*}
	\trace\left\{\sqmatrix{\alpha-1}{0}{0}{I}(A_1+A_2-2A_2^{1/2}QA_1^{1/2})\right\}
	\\
\begin{split}
	=
	~&
	\trace\left\{\sqmatrix{\alpha-1}{0}{0}{I}\big(A_1^{1/2}+A_2^{1/2}-2A_2^{1/2}Q)(A_1^{1/2}-A_2^{1/2})\right\}
	\\
	~&
	+
	\trace\left\{\sqmatrix{\alpha-1}{0}{0}{I}(2A_2-2A_2^{1/2}QA_2^{1/2})\right\}.
\end{split}
\end{multline*}
Since $\alpha\in(0,1)$, by the uniform ellipticity of $A_1$ and $A_2$ we can estimate the first term in the right-hand side by
\begin{equation*}
	4\sqrt{n\Lambda}\,\|A_1^{1/2}-A_2^{1/2}\|.
\end{equation*}
Thus, taking the minimum at both sides of the inequality and recalling \eqref{ctt A tr<0} with $A=A_2$ we get
\begin{multline*}
	\min_{Q\in O(n)}\trace\left\{\sqmatrix{\alpha-1}{0}{0}{I}(A_1+A_2-2A_2^{1/2}QA_1^{1/2})\right\}
	\\
	\leq
	4\left[\sqrt{n\Lambda}\,\|A_1^{1/2}-A_2^{1/2}\|-(1-\alpha)\lambda\right].
\end{multline*}
In consequence, the desired negativity is obtained for small enough $\|A_1^{1/2}-A_2^{1/2}\|$. In particular, imposing the condition
\begin{equation*}
	\|A_1^{1/2}-A_2^{1/2}\|
	=
	\trace\big\{(A_1^{1/2}-A_2^{1/2})^2\big\}
	\leq
	\frac{(1-\alpha)\lambda}{2\sqrt{n\Lambda}},
\end{equation*}
we ensure that
\begin{equation*}
	\min_{Q\in O(n)}\trace\left\{\sqmatrix{\alpha-1}{0}{0}{I}(A_1+A_2-2A_2^{1/2}QA_1^{1/2})\right\}
	\leq
	-2(1-\alpha)\lambda,
\end{equation*}
so \eqref{ineq-Q} follows for every $\alpha\in(0,1)$.

Finally, in the continuous coefficients case, we can always find a small enough $r>0$ such that
\begin{equation*}
	\|A(x)^{1/2}-A(z)^{1/2}\|
	\leq
	\frac{(1-\alpha)\lambda}{2\sqrt{n\Lambda}}
\end{equation*}
for every $x,z\in B_r$.
\end{remark}

\section{ Case $|x_\eta-z_\eta|\leq N\sqrt\lambda\,\varepsilon$}\label{Sec. <Nepsilon}

Remember that for fixed $\eta>0$, we selected $x_\eta$ and $z_\eta$ so that \eqref{sup1} holds. In this section we deal with the case in which the distance between these two points is bounded by $N\sqrt\lambda\,\varepsilon$. The idea is to obtain a contradiction by using two separate arguments depending on how large the distance between $x_\eta$ and $z_\eta$ is. We call these the \emph{medium} and the \emph{short distance case}, and they happen whenever $|x_\eta-z_\eta|>\frac{1}{2}\sqrt\lambda\,\varepsilon$ and $|x_\eta-z_\eta|\leq\frac{1}{2}\sqrt\lambda\,\varepsilon$, respectively.

As we noted in \Cref{sec:keys}, by using the DPP \eqref{DPP} and the counter assumption \eqref{contra}, the inequality \eqref{f < int f} holds for $x_\eta$ and $z_\eta$ satisfying \eqref{sup1}. Hence, for the medium distance case, the contradiction will follow from \Cref{lemma f > int f} that we prove below in \Cref{Sec. medium distance}.

On the other hand, since \Cref{lemma f > int f} is stated for $|x-z|>\frac{1}{2}\sqrt\lambda\,\varepsilon$, and in order to obtain a contradiction, for the short distance case we use a slightly different coupling. We address this case in \Cref{Sec. short distance}.

In contrast to the previous section where $f\equiv f_1$, in this case the key term in our comparison function is the step annular function $f_2$. Hence, in what follows, it will be enough to use the following rough estimate for the term $f_1$ in the comparison function.

\begin{lemma}\label{lemma-estimate-f1}
Let $f_1$ be the function defined in \eqref{f1} and $x,z\in B_r$. Then
\begin{equation}\label{estimate-f1}
				f_1(x+h_x,z+h_z)
				\leq
				f_1(x,z)+3C\Lambda^{\alpha/2}\varepsilon^\alpha
\end{equation}
for every $|h_x|,|h_z|<\sqrt\Lambda\,\varepsilon$.
\end{lemma}

\begin{proof}
First we use the concavity of $t\mapsto t^\alpha$ to estimate
\begin{equation*}
				|(x+h_x)-(z+h_z)|^\alpha-|x-z|^\alpha
				\leq
				|h_x-h_z|^\alpha
				\leq
				2(\sqrt\Lambda\,\varepsilon)^\alpha
\end{equation*}
for every $|h_x|,|h_z|<\sqrt\Lambda\,\varepsilon$, where $\alpha\in(0,1)$ has been used here.
On the other hand, since $x,z\in B_r$, then
\begin{equation*}
\begin{split}
				|(x+h_x)+(z+h_z)|^2-|x+z|^2
				=
				~&
				2(x+z)^\top(h_x+h_z)+|h_x+h_z|^2
				\\
				\leq
				~&
				8r\sqrt\Lambda\,\varepsilon+4\Lambda\,\varepsilon^2
				\\
				\leq
				~&
				12r\sqrt\Lambda\,\varepsilon
				\\
				\leq
				~&
				12r\Lambda^{\alpha/2}\varepsilon^\alpha,
\end{split}
\end{equation*}
where we have recalled that $\sqrt\Lambda\,\varepsilon<\min\{1,r\}$ and $0<\alpha<1$. Thus, recalling \eqref{f1} and choosing large enough
\begin{equation}\label{C3}
				C>12\widetilde Cr,
\end{equation}
we get \eqref{estimate-f1}.
\end{proof}

Before moving into details, observe that, since $n\geq 2$, in this section we can weaken the assumption \eqref{distortion-bound} by imposing the following bounds for the distortion,
\begin{equation}\label{distortion-1-3}
				1\leq\frac{\Lambda}{\lambda}\leq 3,
\end{equation}
with no dependence on the dimension. Hence, recalling \eqref{inclusions}, the inclusions
\begin{equation}\label{inclusions-1-3}
				\B
				\subset
				\frac{1}{\sqrt\lambda}\,A(x)^{1/2}\B
				\subset
				\sqrt3\,\B
\end{equation}
hold uniformly for every $x\in\Omega$.

\subsection{The medium distance case: $\frac{1}{2}\,\sqrt\lambda\,\varepsilon<|x_\eta-z_\eta|
				\leq
				 N\sqrt\lambda\,\varepsilon$}\label{Sec. medium distance}
In this section, we prove  \Cref{lemma f > int f} in the case  $\frac{1}{2}\,\sqrt\lambda\,\varepsilon<|x-z|
				\leq
				 N\sqrt\lambda\,\varepsilon$.

Recalling definition of $f_2$ in \eqref{f2}, this function can be expressed as a finite disjoint sum
\begin{equation}\label{f2-sum}
				f_2(x,z)
				=
				\sum_{i=0}^{2N}C^{2(2N-i)}\varepsilon^\alpha \chi_{S_i}(x,z),
\end{equation}
where
\begin{equation}\label{S_i}
				S_i
				:\,=
				\Big\{(x,z)\in\R^{2n} \, : \, \frac{i-1}{2}<\frac{|x-z|}{\sqrt\lambda\,\varepsilon}\leq \frac{i}{2}\Big\}
\end{equation}
for each $i=0,1,\ldots,2N$ and $\chi_{S_i}(x,z)$ stands for the function which is equal to $1$ whenever $(x,z)\in S_i$ and $0$ otherwise. 
 By the assumptions of this section, we can fix $j=2,3,\ldots,2N$ so that $(x,z)$ belongs to $S_j$. The next is the main result of this section. 
 
\begin{lemma}\label{lemma-ineq-f2}
Let $f_2$ be the function defined in \eqref{f2} and $x,z\in B_r$ such that $(x-z)/|x-z|=e_1$ and
\begin{equation*}
				\frac{1}{2}\,\sqrt\lambda\,\varepsilon
				<
				|x-z|
				\leq
				N\sqrt\lambda\,\varepsilon.
\end{equation*}
Define the vectors $\nu_x=A(x)^{-1/2}e_1$ and $\nu_z=A(z)^{-1/2}e_1$ and fix an orthogonal matrix $Q\in O(n)$ such that
\begin{equation}\label{Qvxvz}
				Q\,\frac{\nu_x}{|\nu_x|}
				=
				-\frac{\nu_z}{|\nu_z|}.
\end{equation}
Then
\begin{equation}\label{ineq-f2}
				\dashint_\B f_2\big(x+\varepsilon A(x)^{1/2}\,y,z+\varepsilon A(z)^{1/2}Q\,y\big)\ dy
				\geq
				\gamma\,C^2 f_2(x,z),
\end{equation}
where $\gamma\in(0,1)$ is a fixed constant depending only on $n$.
\end{lemma}

\begin{proof}
For the sake of simplicity, let us write $A_1=A(x)$ and $A_2=A(z)$. Similarly as in \eqref{f2-sum}, given $A_1,A_2\in\A(\lambda,\Lambda)$ and $Q\in O(n)$ we write
\begin{equation*}
				f_2\big(x+\varepsilon A_1^{1/2}\,y,z+\varepsilon A_2^{1/2}Q\,y\big)
				=
				\sum_{i=0}^{2N}C^{2(2N-i)}\varepsilon^\alpha \chi_{S_i^*}(y),
\end{equation*}
where
\begin{equation*}
				S_i^*
				:\,=
				\big\{y\in\B \, : \, \big(x+\varepsilon A_1^{1/2}\,y,z+\varepsilon A_2^{1/2}Q\,y\big)\in S_i\big\},
\end{equation*}
for $i=0,1,\ldots,2N$.

Therefore, computing the averaged integral we get that
\begin{multline*}
				\dashint_\B f_2\big(x+\varepsilon A_1^{1/2}\,y,z+\varepsilon A_2^{1/2}Q\,y\big)\ dy
				\\
				=
				\frac{1}{|\B|}\sum_{i=0}^{2N}C^{2(2N-i)}\varepsilon^\alpha |S_i^*|
				\geq
				\frac{1}{|\B|}\sum_{i=0}^{j-1}C^{2(2N-i)}\varepsilon^\alpha |S_i^*|
				\\
				\geq
				C^{2(2N-j+1)}\varepsilon^\alpha\,\frac{1}{|\B|}\sum_{i=0}^{j-1} |S_i^*|
				=
				C^2f_2(x,z)\,\frac{1}{|\B|}\sum_{i=0}^{j-1} |S_i^*|,
\end{multline*}
where in the first inequality we have used that $f_2\geq 0$ to discard all terms in the sum with $i\geq j$ and in the second inequality we have used that $C^{-2i}\geq C^{-2(j-1)}$ for $i=0,1,\ldots,j-1$. Therefore, it only remains to show that the sum in the right-hand side of the previous estimate has a lower bound depending only on $n$, $\lambda$ and $\Lambda$. Recalling the definition of the sets $S_i^*$ together with the fact that they are pairwise disjoint we obtain
\begin{multline*}
				\frac{1}{|\B|}\sum_{i=0}^{j-1}\abs{S_i^*}
				=
				\frac{1}{|\B|}\abs{\Big\{y\in\B \, : \, \big|x-z+\varepsilon\big(A_1^{1/2}-A_2^{1/2}Q\big)y\big|\leq\frac{j-1}{2}\sqrt\lambda\,\varepsilon\Big\}}
				\\
				=
				\frac{1}{|\B|}\abs{\Big\{y\in\B \, : \, \abs{\frac{|x-z|}{\sqrt\lambda\,\varepsilon}\,e_1+ \big(I-A_2^{1/2}QA_1^{-1/2}\big)\frac{1}{\sqrt\lambda}\,A_1^{1/2}y}\leq\frac{j-1}{2}\Big\}},
\end{multline*}
where the hypothesis $(x-z)/|x-z|=e_1$ has been used here. Note that, letting $w=\lambda^{-1/2}A_1^{1/2}y$ and recalling \eqref{inclusions-1-3} we obtain that
\begin{multline*}
				\frac{1}{|\B|}\sum_{i=0}^{j-1}\abs{S_i^*}
				\\
\begin{split}
				=
				~&
				\frac{1}{|\lambda^{-1/2}A_1^{1/2}\B|}\abs{\Big\{w\in\lambda^{-1/2}A_1^{1/2}\B \, : \, \abs{\frac{|x-z|}{\sqrt\lambda\,\varepsilon}\,e_1+ \big(I-A_2^{1/2}QA_1^{-1/2}\big)w}\leq\frac{j-1}{2}\Big\}},
				\\
				\geq
				~&
				3^{-n/2}\frac{1}{\abs{\B}}\,\abs{\Big\{w\in\B \, : \, \abs{\frac{|x-z|}{\sqrt\lambda\,\varepsilon}\,e_1+ \big(I-A_2^{1/2}QA_1^{-1/2}\big)w}\leq\frac{j-1}{2}\Big\}}.
\end{split}
\end{multline*}
Next, we claim that there exists small enough fixed constant $\varrho>0$ so that the following inequality holds
\begin{equation}\label{aim-f2}
				\abs{\frac{|x-z|}{\sqrt\lambda\,\varepsilon}\,e_1+ \big(I-A_2^{1/2}QA_1^{-1/2}\big)w}
				\leq
				\frac{j-1}{2}
				\quad \mbox{for every }
				w\in B_\varrho(-\tfrac{1}{3}\, e_1).
\end{equation}
Indeed, assuming \eqref{aim-f2} we have that
\begin{equation*}
				\frac{1}{\abs{\B}}\,\abs{\Big\{w\in\B \, : \, \abs{\frac{|x-z|}{\sqrt\lambda\,\varepsilon}\,e_1+ \big(I-A_2^{1/2}QA_1^{-1/2}\big)w}\leq\frac{j-1}{2}\Big\}}
				\geq
				\varrho^n
\end{equation*}
and \eqref{ineq-f2} follows from this and the previous estimates with $\gamma=3^{-n/2}\varrho^n$.

To prove that there exists $\varrho>0$ such that \eqref{aim-f2} holds, let $Q$ be an orthogonal matrix satisfying \eqref{Qvxvz}, then
\begin{equation*}
				A_2^{1/2}QA_1^{-1/2}e_1
				=
				-\frac{|\nu_x|}{|\nu_z|}\,e_1.
\end{equation*}
Observe that by \eqref{distortion-1-3} and the definition of $\nu_x$ and $\nu_z$ we have in particular that
\begin{equation}\label{distortion-nu}
				\frac{1}{\sqrt3}
				\leq
				\frac{|\nu_x|}{|\nu_z|}
				\leq
				\sqrt3.
\end{equation}
Now fix any $w=-\frac{1}{3}\,e_1+\varrho\zeta$ with $\zeta\in\B$ and insert it in the left-hand side of \eqref{aim-f2} to get
\begin{equation*}
				\abs{\frac{|x-z|}{\sqrt\lambda\,\varepsilon}\,e_1+ \big(I-A_2^{1/2}QA_1^{-1/2}\big)w}
				\leq
				\abs{\frac{|x-z|}{\sqrt\lambda\,\varepsilon}-\frac{1}{3}\pare{1+\frac{|\nu_x|}{|\nu_z|}}}+3\varrho,
\end{equation*}
where the second term in the right-hand side of this inequality follows from the rough estimate
\begin{equation*}
				\abs{\zeta-A_2^{1/2}QA_1^{-1/2}\zeta}
				\leq 1+\sqrt3
				<
				3,
\end{equation*}
which holds uniformly for every $A_1,A_2\in\A(\lambda,3\lambda)$ and $\zeta\in\B$.
Next, recalling that $(x,z)\in S_j$ with $j\geq 2$, we use the condition in the definition of the set $S_j$ in \eqref{S_i} together with \eqref{distortion-nu} to see that
\begin{enumerate}[label=\ (\roman*),align=left,leftmargin=0pt]\setlength\itemsep{3pt}
\item if $\displaystyle\frac{|x-z|}{\sqrt\lambda\,\varepsilon}>\frac{1}{3}\pare{1+\frac{|\nu_x|}{|\nu_z|}}$, then
\begin{equation*}
\begin{split}
				\abs{\frac{|x-z|}{\sqrt\lambda\,\varepsilon}-\frac{1}{3}\pare{1+\frac{|\nu_x|}{|\nu_z|}}}-\frac{j-1}{2}
				\leq
				~&
				\frac{j}{2}-\frac{1}{3}\pare{1+\frac{|\nu_x|}{|\nu_z|}}-\frac{j-1}{2}
				\\
				=
				~&
				\frac{1}{6}-\frac{1}{3}\cdot\frac{|\nu_x|}{|\nu_z|}
				\\
				\leq
				~&
				\frac{1}{6}-\frac{1}{3\sqrt3}
				<
				-\frac{1}{100},
\end{split}
\end{equation*}
\item otherwise, if $\displaystyle\frac{|x-z|}{\sqrt\lambda\,\varepsilon}\leq\frac{1}{3}\pare{1+\frac{|\nu_x|}{|\nu_z|}}$, then
\begin{equation*}
				\abs{\frac{|x-z|}{\sqrt\lambda\,\varepsilon}-\frac{1}{3}\pare{1+\frac{|\nu_x|}{|\nu_z|}}}-\frac{j-1}{2}
				<
				\frac{1}{3}\pare{1+\frac{|\nu_x|}{|\nu_z|}}-\frac{j-1}{2}-\frac{j-1}{2}
				\leq
				\frac{4}{3}+\frac{1}{\sqrt3}-j.
\end{equation*}
Observe that the last term is strictly less than $2-j$, so the negativity of this term is ensured since $j\geq 2$.
\end{enumerate}
Therefore, we have shown that if $(x,z)\in S_j$ for some $j\geq 2$, then 
\begin{equation*}
				\abs{\frac{|x-z|}{\sqrt\lambda\,\varepsilon}-\frac{1}{3}\pare{1+\frac{|\nu_x|}{|\nu_z|}}}
				<
				\frac{j-1}{2}-\frac{1}{100}.
\end{equation*}
Finally, choosing $\varrho=\frac{1}{300}$, the inequality \eqref{aim-f2} holds for every $w\in B_{1/300}(-\frac{1}{3}e_1)$. This completes the proof of the lemma.
\end{proof}

Hence, since $f=f_1-f_2$, combining \Cref{lemma-estimate-f1,lemma-ineq-f2} we get
\begin{multline*}
				\dashint_\B f\big(x+\varepsilon A(x)^{1/2}\,y,z+\varepsilon A(z)^{1/2}Q\,y\big)\ dy -f(x,z)
				\\
				\leq
				3C\Lambda^{\alpha/2}\varepsilon^\alpha-(\gamma\,C^2-1) f_2(x,z).
\end{multline*}
Now, by definition of $f_2$, it turns out that $f_2(x,z)\geq \varepsilon^\alpha$ for every $|x-z|\leq N\sqrt\lambda\,\varepsilon$, and thus choosing large enough $C$ such that
\begin{equation}\label{C4}
				\gamma\, C^2-3C\Lambda^{\alpha/2}-2
				>
				0,
\end{equation}
we ensure that the right-hand side of the previous inequality is less than $-\varepsilon^\alpha<-\varepsilon^2$. Then \eqref{f > int f} follows, so the proof of \Cref{lemma f > int f} is complete in the case $|x-z|\leq N\sqrt\lambda\,\varepsilon$.

\subsection{The short distance case: $|x_\eta-z_\eta|
				\leq
				\frac{1}{2}\,\sqrt\lambda\,\varepsilon$}\label{Sec. short distance}

In contrast with the large distance case and the medium distance case addressed in \Cref{Sec. >Nepsilon} and \Cref{Sec. medium distance}, respectively, where the contradiction was obtained by proving \Cref{lemma f > int f}, in this section we thrive for a contraction in a slightly different way. The main difference is that, in this case, we take an advantage from the full cancellation effect that happens when the two points $x_\eta$ and $z_\eta$ are close enough.

Recall that $x_\eta,z_\eta\in B_r$ satisfy \eqref{sup1}.
Observe that, since
\begin{equation*}
				B_{\sqrt\lambda\,\varepsilon}(x)
				\subset
				E_x
\end{equation*}
by uniform ellipticity, if the distance between $x_\eta$ and $z_\eta$ is small enough, then the corresponding ellipsoids $E_{x_\eta}$ and $E_{z_\eta}$ have non-empty intersection. Indeed, since
$
				0
				<
				|x_\eta-z_\eta|
				\leq
				\frac{1}{2}\,\sqrt\lambda\,\varepsilon
$
by hypothesis, it is easy to check that
\begin{equation*}
				B_{\sqrt\lambda\,\varepsilon/4}\pare{\frac{x_\eta+z_\eta}{2}}\subset E_{x_\eta}\cap E_{z_\eta},
\end{equation*}
and thus
\begin{equation}\label{intersection-ineq}
				\frac{|E_{x_\eta}\cap E_{z_\eta}|}{|E_{x_\eta}|}
				\geq
				\gamma(n,\lambda,\Lambda)
				:\,=
				\bigg(\frac{1}{4}\sqrt{\frac{\lambda}{\Lambda}}\bigg)^n.
\end{equation}
This means that, in the $2n$-dimensional process described in \Cref{Sec. Ellipsoid Process}, starting from $(x_\eta,z_\eta)$ we can choose a point $y\in\R^n$ belonging to both $E_{x_\eta}$ and $E_{z_\eta}$. This choice gives a full cancellation but we also need that the measure of the ellipsoids is the same, that is, $\det\{A(x_\eta)\}=\det\{A(z_\eta)\}$. In fact, this is the only place in the paper where such an assumption is needed.

Then we repeat the argument in \eqref{eq:key-contra} but with a different coupling taking advantage of the full cancellation where the ellipsoids overlap. Starting from \eqref{sup1} and using the DPP \eqref{DPP} for estimating the difference $u(x_\eta)-u(z_\eta)$ we get
\begin{equation}\label{short-distance}
\begin{split}
				f(x_\eta,z_\eta)-\eta
				\leq
				~&
				u(x_\eta)-u(z_\eta)-K
				=
				\dashint_{E_{x_\eta}} u(\zeta)\ d\zeta - \dashint_{E_{z_\eta}} u(\xi)\ d\xi-K
				\\
				=
				~&
				\frac{1}{|E_{x_\eta}|}\brackets{\int_{E_{x_\eta}\setminus E_{z_\eta}}u(\zeta)\ d\zeta-\int_{E_{z_\eta}\setminus E_{x_\eta}}u(\xi)\ d\xi}-K
				\\
				=
				~&
				\frac{|E_{x_\eta}\setminus E_{z_\eta}|}{|E_{x_\eta}|}\,\dashint_{E_{x_\eta}\setminus E_{z_\eta}}\dashint_{E_{z_\eta}\setminus E_{x_\eta}} \big[u(\zeta)-u(\xi)\big]\ d\xi\,d\zeta-K
				\\
				\leq
				~&
				\frac{|E_{x_\eta}\setminus E_{z_\eta}|}{|E_{x_\eta}|}\brackets{K+\dashint_{E_{x_\eta}\setminus E_{z_\eta}}\dashint_{E_{z_\eta}\setminus E_{x_\eta}} f(\zeta,\xi)\ d\xi\,d\zeta}-K
				\\
				=
				~&
				\frac{|E_{x_\eta}\setminus E_{z_\eta}|}{|E_{x_\eta}|}\,\dashint_{E_{x_\eta}\setminus E_{z_\eta}}\dashint_{E_{z_\eta}\setminus E_{x_\eta}} f(\zeta,\xi)\ d\xi\,d\zeta-\frac{|E_{x_\eta}\cap E_{z_\eta}|}{|E_{x_\eta}|}\,K
				\\
				\leq
				~&
				\dashint_{E_{x_\eta}\setminus E_{z_\eta}}\dashint_{E_{z_\eta}\setminus E_{x_\eta}} f(\zeta,\xi)\ d\xi\,d\zeta-\gamma K,
\end{split}
\end{equation}
where \eqref{ineq K B_2r} has been recalled in the second inequality and in the last inequality we have used that $|E_{x_\eta}\setminus E_{z_\eta}|<|E_{x_\eta}|$ and \eqref{intersection-ineq}. Since $f=f_1-f_2\leq f_1$ and $E_x\subset B_{\sqrt\Lambda\,\varepsilon}(x)$ by uniform ellipticity, we can estimate the integral part in the right-hand side above using the estimate for $f_1$ from \Cref{lemma-estimate-f1}:
\begin{equation*}
\begin{split}
				\dashint_{E_{x_\eta}\setminus E_{z_\eta}}\dashint_{E_{z_\eta}\setminus E_{x_\eta}} f(\zeta,\xi)\ d\xi\,d\zeta
				\leq
				~&
				\dashint_{E_{x_\eta}\setminus E_{z_\eta}}\dashint_{E_{z_\eta}\setminus E_{x_\eta}} f_1(\zeta,\xi)\ d\xi\,d\zeta
				\\
				\leq
				~&
				f_1(x_\eta,z_\eta)+3C\Lambda^{\alpha/2}\varepsilon^\alpha.
\end{split}
\end{equation*}
For the remaining term, recalling the definition of $f_2$ in \eqref{f2} together with the fact that $0<|x_\eta-z_\eta|\leq\frac{1}{2}\,\sqrt\lambda\,\varepsilon$, we have that $f_2(x_\eta,z_\eta)=C^{2(2N-1)}\varepsilon^\alpha$. On the other hand, combining this with the counter assumption \eqref{contra} we get
\begin{equation*}
				K
				>
				C^{4N}\varepsilon^\alpha
				=
				C^2f_2(x_\eta,z_\eta).
\end{equation*} 
Hence, putting all these estimates together in \eqref{short-distance} we have that
\begin{equation*}
				f(x_\eta,z_\eta)-\eta
				\leq
				f_1(x_\eta,z_\eta)+3C\Lambda^{\alpha/2}\varepsilon^\alpha-\gamma C^2f_2(x_\eta,z_\eta).
\end{equation*}
Let $C>\gamma^{-1/2}$. Since $f=f_1-f_2$ and $f_2(x_\eta,z_\eta)\geq\varepsilon^\alpha$, then
\begin{equation*}
				(\gamma C^2-1)\,\varepsilon^\alpha
				\leq
				3C\Lambda^{\alpha/2}\varepsilon^\alpha+\eta.
\end{equation*}
Finally, we obtain a contradiction with this inequality by letting $\eta=\varepsilon^\alpha$ and choosing large enough $C$ so that
\begin{equation*}
				\gamma\, C^2-3C\Lambda^{\alpha/2}-2
				>
				0.
\end{equation*}
Therefore, this implies the falseness of the counter assumption \eqref{contra}. Thus \eqref{aim} holds and the proof of \Cref{MAIN-THM} is finished.

\section{Examples}\label{Sec. Examples}

\subsection{Diagonal case}

As it has been proved in the previous sections, the inequality \eqref{distortion-bound-alpha} provides a sufficient condition for \eqref{f > int f}, from which the asymptotic H\"older regularity of solutions $u_\varepsilon$ to the DPP \eqref{DPP} follows. One of the key results for deriving this condition is \Cref{lemma-max-trace}, which states that, given a matrix $M\in\R^{n\times n}$, the following holds,
\begin{equation*}
				\max_{Q\in O(n)}\trace\{MQ\}
				=
				\trace\{(M^\top\! M)^{1/2}\}
				\geq
				n\left|\det\{M\}\right|^{1/n}.
\end{equation*}
Then the idea of the proof consisted basically in obtaining a lower estimate for the right-hand side in the previous formula (see the proof of \Cref{lemma-min-trace}). However, we could try to obtain the desired estimate for the left-hand side directly without using the inequality above. Therefore, using the equality $\displaystyle\max_{Q\in O(n)}\trace\{MQ\}=\trace\{(M^\top\!M)^{1/2}\}$ we obtain
\begin{multline*}
				\min_{Q\in O(n)}\trace\left\{\sqmatrix{\alpha-1}{0}{0}{I}(A_1+A_2-2A_2^{1/2}QA_1^{1/2})\right\}
				\\
				=
				\trace\left\{\sqmatrix{\alpha-1}{0}{0}{I}(A_1+A_2)-2\bigg[\sqmatrix{\alpha-1}{0}{0}{I}A_1\sqmatrix{\alpha-1}{0}{0}{I}A_2\bigg]^{1/2}\right\},
\end{multline*}
where $A_1,A_2\in\A(\lambda,\Lambda)$ and $\alpha\in(0,1)$.
Then a stronger sufficient condition for \eqref{f > int f} to be satisfied is that the matrix inequality
\begin{multline}\label{trace-inequality}
				\frac{1}{2}\trace\left\{\sqmatrix{\alpha-1}{0}{0}{I}(A_1+A_2)\right\}
				\\
				<
				\trace\left\{\bigg[\sqmatrix{\alpha-1}{0}{0}{I}A_1\sqmatrix{\alpha-1}{0}{0}{I}A_2\bigg]^{1/2}\right\}
\end{multline}
holds for every pair of matrices $A_1,A_2\in\A(\lambda,\Lambda)$. Inequality \eqref{trace-inequality} can be seen as some sort of reverse inequality of arithmetic
and geometric means for traces.
Therefore, it seems reasonable to expect that the assumption on the bound of the distortion of the ellipsoids \eqref{distortion-bound-alpha} can be weakened. For example, let us consider the case in which the axes of the ellipsoids are aligned with the coordinate axes. This is the case in which both $A_1$ and $A_2$ are diagonal matrices in $\A(\lambda,\Lambda)$. We denote by $\lambda_1(A_1),\lambda_2(A_1),\ldots,\lambda_n(A_1)$ and $\lambda_1(A_2),\lambda_2(A_2),\ldots,\lambda_n(A_2)$ the elements in the diagonals of of $A_1$ and $A_2$, respectively. Recall that all these elements coincide with the eigenvalues of $A_1$ and $A_2$ and thus they are bounded between $\lambda$ and $\Lambda$.

\begin{lemma}
The inequality \eqref{trace-inequality} holds for every diagonal matrices $A_1$ and $A_2$ in $\A(\lambda,\Lambda)$ if and only if
\begin{equation}\label{distortion-bound-diagonal-case}
				1
				\leq
				\frac{\Lambda}{\lambda}
				<
				\bigg(1+2\sqrt{\frac{1-\alpha}{n-1}}\bigg)^2.
\end{equation}
\end{lemma}

\begin{proof}
Inserting $A_1=\mathrm{diag}(\lambda_1(A_1),\ldots,\lambda_n(A_1))$ and $A_2=\mathrm{diag}(\lambda_1(A_2),\ldots,\lambda_n(A_2))$ in \eqref{trace-inequality} we get
\begin{multline*}
				-(1-\alpha)\big(\lambda_1(A_1)+\lambda_1(A_2)\big)+\sum_{i=2}^n\big(\lambda_i(A_1)+\lambda_i(A_2)\big)
				\\
				<
				2(1-\alpha)\sqrt{\lambda_1(A_1)\lambda_1(A_2)\,}+2\sum_{i=2}^n\sqrt{\lambda_i(A_1)\lambda_i(A_2)\,}.
\end{multline*}
Rearranging the terms we obtain
\begin{equation}\label{diagonal-ineq}
				\sum_{i=2}^n\big(\sqrt{\lambda_i(A_1)}-\sqrt{\lambda_i(A_2)}\big)^2
				<
				(1-\alpha)\big(\sqrt{\lambda_1(A_1)}+\sqrt{\lambda_1(A_2)}\big)^2.
\end{equation}
Now observe that, by uniform ellipticity,
\begin{equation*}
				\sum_{i=2}^n\big(\sqrt{\lambda_i(A_1)}-\sqrt{\lambda_i(A_2)}\big)^2
				\leq
				(n-1)(\sqrt{\Lambda}-\sqrt{\lambda})^2
\end{equation*}
and
\begin{equation*}
				(1-\alpha)\big(\sqrt{\lambda_1(A_1)}+\sqrt{\lambda_1(A_2)}\big)^2
				\geq
				4(1-\alpha)\lambda.
\end{equation*}
Moreover, the equality is attained in both inequalities for $A_1=\lambda\,I$ and $A_2$ the diagonal matrix with entries $\lambda,\Lambda,\Lambda,\ldots,\Lambda$. Thus, \eqref{diagonal-ineq} holds for every diagonal matrices $A_1$ and $A_2$ in $\A(\lambda,\Lambda)$ if and only if
\begin{equation*}
				(n-1)(\sqrt{\Lambda}-\sqrt{\lambda})^2
				<
				4(1-\alpha)\lambda.
\end{equation*}
Hence, the bound in \eqref{distortion-bound-diagonal-case} follows.
\end{proof}

Observe that the bound in \eqref{distortion-bound-diagonal-case} is larger than in \eqref{distortion-bound-alpha}. However, in this result we do not consider the whole family of ellipsoids, 
only the ones aligned with the coordinate axes, and hence we can not say that \eqref{distortion-bound-diagonal-case} is a sufficient condition for asymptotic H\"older regularity.

\subsection{The mirror point coupling}
This work is related to the results in \cite{lindvallr86} except that our setting is discrete, we do not require continuity for the ellipsoids, and our proof is written in purely analytic form i.e.\ we avoid the probabilistic tools. The Lindvall--Rogers coupling (see eq.\ (2) in their paper), can be written without the drift as 
\begin{align*}
dX_t&=\sigma(X_t)\ dB_t,\\
dX'_t&=\sigma(X'_t)H(X_t, X'_t)\ dB_t,
\end{align*}
where $B$ is the Brownian motion.
If we formally discretize this, we get
\begin{align*}
X_{k+1} &= X_k + \sigma(X_k)U_k,\\
X'_{k+1} &=X'_{k+1}+\sigma(X'_k)H(X_k ,X'_k) U_k,
\end{align*}
where $U_1,U_2,U_3,\ldots$ is an i.i.d. sequence of uniform random variables on the ball $B_ \eps$. Above, $H(X_k, X'_k)$ is an orthogonal reflection matrix that can depend on the two coupled walks. Observe that our discrete results imply similar regularity results for the limit. Furthermore, the discrete results can also be interesting from the point of view of discretization of PDEs.

On the other hand, it is known that the so-called mirror point reflection coupling can be used to obtain Lipschitz estimates for functions satisfying the mean value property with balls. If $x$ and $z$ denote the centers of two balls, this coupling consists on reflecting the points of one of the balls into the other with respect to the hyperplane orthogonal to $z-x$ passing through $(x+z)/2$. Therefore, it is reasonable to expect that the same type of coupling should work if the balls are replaced by ellipsoids with sufficiently small distortion. In this subsection we adapt the mirror point coupling to our setting and we study under which conditions this coupling works. Moreover, we compare it with the optimal coupling arising from the minimization of \eqref{Q-map}.

Recalling \Cref{optimal-orthogonal-coupling}, the optimal coupling matrix $Q$ for which the minimum of \eqref{Q-map} is attained depends on $A_1$, $A_2$ and $\alpha$, and can be explicitly written as follows,
\begin{equation*}
	Q
	=
	\left[\left(A_1^{1/2}\sqmatrix{\alpha-1}{0}{0}{I}A_2^{1/2}\right)^2\right]^{-1/2}A_1^{1/2}\sqmatrix{\alpha-1}{0}{0}{I}A_2^{1/2}.
\end{equation*}
Observe that $Q$ is indeed an orthogonal matrix with determinant $-1$, so $Q$ can be also seen as a reflection in the Euclidean $n$-dimensional space. For this choice of $Q$, \eqref{ineq-Q} holds if the condition \eqref{distortion-bound-alpha} is satisfied.

However, it is possible to fix $Q$ without any dependence on the matrices or the exponent in such a way that \eqref{ineq-Q} still holds under certain restriction on the distortion. More precisely, this choice corresponds to the reflection matrix on the first coordinate axis, that is
\begin{equation*}
	Q
	=
	J_0
	:\,=
	\sqmatrix{-1}{0}{0}{I}.
\end{equation*}
To see this, replacing $\sqmatrix{\alpha-1}{0}{0}{1}=J_0+\alpha e_1e_1^\top$ we obtain
\begin{multline*}
	\trace\left\{\sqmatrix{\alpha-1}{0}{0}{I}(A_1+A_2-2A_2^{1/2}QA_1^{1/2})\right\}
	\\
	=
	\trace\big\{J_0\big(A_1+A_2-2A_2^{1/2}J_0A_1^{1/2})\big\}
	+
	\alpha\big|\big(A_1^{1/2}-A_2^{1/2}J_0\big)^\top e_1\big|^2.
\end{multline*}
By uniform ellipticity we can easily obtain a bound for the second term,
\begin{equation*}
	\alpha\big|\big(A_1^{1/2}-A_2^{1/2}J_0\big)^\top e_1\big|^2
	\leq
	\alpha\big(|A_1^{1/2}e_1|+|A_2^{1/2}e_1|\big)^2
	\leq
	4\alpha\Lambda,
\end{equation*}
while for the other term, observing that $J_0A_2^{1/2}J_0\in\mathcal{A}(\sqrt{\lambda},\sqrt{\Lambda})$, we get that
\begin{equation*}
\begin{split}
	\trace\big\{J_0\big(A_1+A_2-2A_2^{1/2}J_0A_1^{1/2})\big\}
	=
	~&
	\trace\big\{J_0(A_1+A_2)\big\}-2\trace\big\{J_0A_2^{1/2}J_0A_1^{1/2}\big\}
	\\
	\leq
	~&
	2\big[-\lambda+(n-1)\Lambda\big]-2n\lambda
	\\
	=
	~&
	2\big[(n-1)\Lambda-(n+1)\lambda\big].
\end{split}
\end{equation*}
Putting together these estimates we have
\begin{equation*}
	\trace\left\{\sqmatrix{\alpha-1}{0}{0}{I}(A_1+A_2-2A_2^{1/2}QA_1^{1/2})\right\}
	\leq
	2\big[(n-1+2\alpha)\Lambda-(n+1)\lambda\big],
\end{equation*}
so \eqref{ineq-Q} holds with $Q=J_0$ if
\begin{equation}\label{dist-mirror}
	\frac{\Lambda}{\lambda}
	<
	\frac{n+1}{n-1+2\alpha}.
\end{equation}

Comparing this with \eqref{distortion-bound-alpha}, we observe that \eqref{dist-mirror} improves the exponent $\alpha$ when the distortion $\Lambda/\lambda$ is close enough to $1$. However, for sufficiently large distortion (but still below $\frac{n+1}{n-1}$), \eqref{distortion-bound-alpha} gives a better exponent. In any case, both conditions give the same upper bound for the distortion when $\alpha\to 0$, that is $\Lambda/\lambda<\frac{n+1}{n-1}$ for every $\alpha\in(0,1)$.

It is worth noting that, since $Q$ was originally chosen to be the orthogonal matrix minimizing the left-hand side in \eqref{ineq-Q}, it is reasonable to expect that a refinement of the argument in \Cref{lemma-min-trace} would produce better estimates, as it was already pointed out in the discussion at the beginning of \Cref{Sec. Examples}. In this direction, in the following example we show (in the constant coefficients case, $A_1=A_2$) that, even when the distortion $\Lambda/\lambda$ is small, the optimal choice of $Q$ minimizing \eqref{Q-map} differs from the reflection $J_0$.

\begin{example}
In this example we discuss the effect of choosing the mirror point coupling in \eqref{ineq-Q} when $A_1=A_2$ is certain $2\times 2$ symmetric matrix with given distortion $\Lambda/\lambda=\delta\geq 1$. In particular, we compare the conditions under which \eqref{ineq-Q} holds when $Q$ is either $J_0$ or the optimal coupling minimizing \eqref{Q-map}.

Let $A$ be a $2\times 2$ symmetric and positive definite matrix. Then the inequality \eqref{ineq-Q} for $A_1=A_2=A$ is equivalent to
\begin{equation}\label{n=2}
	\trace\big\{A^{1/2}J_\alpha A^{1/2}(I-Q)\big\}
	<
	0,
\end{equation}
where we have denoted
\begin{equation*}
	J_\alpha
	=
	\sqmatrix{\alpha-1}{0}{0}{1}
\end{equation*}
for simplicity.
Fix $n=2$ and
\begin{equation*}
	A
	=
	\sqmatrix{\delta+1}{\delta-1}{\delta-1}{\delta+1},
\end{equation*}
with $\delta\geq 1$. Then $A$ has eigenvalues $\lambda=2$ and $\Lambda=2\delta$, so its distortion is $\Lambda/\lambda=\delta$. In this example we compare the conditions under which \eqref{n=2} holds when $Q$ is either $J_0$ or the nearest orthogonal matrix to $A^{1/2}J_\alpha A^{1/2}$, for which the minimum is attained.

First, if $Q=J_0$, then $I-J_0=\sqmatrix{2}{0}{0}{0}$ and thus
\begin{equation*}
	\trace\big\{A^{1/2}J_\alpha A^{1/2}(I-J_0)\big\}
	=
	2e_1^\top A^{1/2}J_\alpha A^{1/2}e_1.
\end{equation*}
Computing the square root of $A$ we get
\begin{equation*}
	A^{1/2}
	=
	\frac{1}{\sqrt{2}}\sqmatrix{\sqrt{\delta}+1}{\sqrt{\delta}-1}{\sqrt{\delta}-1}{\sqrt{\delta}+1},
\end{equation*}
so
\begin{equation*}
	\trace\big\{A^{1/2}J_\alpha A^{1/2}(I-J_0)\big\}
	=
	-(1-\alpha)(\sqrt{\delta}+1)^2+(\sqrt{\delta}-1)^2.
\end{equation*}
Hence, \eqref{n=2} holds with $Q=J_0$ if and only if
\begin{equation*}
	\alpha
	<
	1-\bigg(\frac{\sqrt{\delta}-1}{\sqrt{\delta}+1}\bigg)^2,
\end{equation*}
or equivalently
\begin{equation*}
	\frac{\Lambda}{\lambda}
	=
	\delta
	<
	\bigg(\frac{1+\sqrt{1-\alpha}}{1-\sqrt{1-\alpha}}\bigg)^2.
\end{equation*}

On the other hand, by the equality in \eqref{max-trace} from \Cref{lemma-max-trace} with $M=A^{1/2}J_\alpha A^{1/2}$ we have that
\begin{equation*}
	\min_{Q\in O(n)}\trace\big\{A^{1/2}J_\alpha A^{1/2}(I-Q)\big\}
	=
	\trace\big\{A^{1/2}J_\alpha A^{1/2}\big\}-\trace\big\{\big[(A^{1/2}J_\alpha A^{1/2})^2\big]^{1/2}\big\}.
\end{equation*}
It is easy to check that
\begin{equation*}
	\trace\big\{A^{1/2}J_\alpha A^{1/2}\big\}
	=
	\trace\big\{J_\alpha A\big\}
	=
	\alpha(\delta+1).
\end{equation*}
For the other term we recall that
\begin{equation*}
	\trace\{M^{1/2}\}^2
	=
	\trace\{M\}+2\sqrt{\det\{M\}}
\end{equation*}
for every positive definite symmetric $2\times 2$ matrix $M$, so
\begin{equation*}
	\trace\big\{\big[(A^{1/2}J_\alpha A^{1/2})^2\big]^{1/2}\big\}
	=
	\sqrt{\trace\big\{(J_\alpha A)^2\big\}+2\sqrt{\det\{(J_\alpha A)^2\}}},
\end{equation*}
where
\begin{equation*}
	\trace\{(J_\alpha A)^2\}
	=
	\alpha^2(\delta+1)^2+8(1-\alpha)\delta
\end{equation*}
and
\begin{equation*}
	\sqrt{\det\{(J_\alpha A)^2\}}
	=
	(1-\alpha)\det\{A\}
	=
	4(1-\alpha)\delta.
\end{equation*}
Hence, replacing above we obtain
\begin{equation*}
	\min_{Q\in O(n)}\trace\big\{A^{1/2}J_\alpha A^{1/2}(I-Q)\big\}
	=
	\alpha(\delta+1)-\sqrt{\alpha^2(\delta+1)^2+16(1-\alpha)\delta},
\end{equation*}
which is clearly always negative, independently of the relation between the distortion $\Lambda/\lambda=\delta$ and the exponent $\alpha$.

In conclusion:
\begin{equation*}
	\trace\big\{A^{1/2}J_\alpha A^{1/2}(I-J_0)\big\}
	<
	0
	\qquad\Longleftrightarrow\qquad
	1
	\leq
	\frac{\Lambda}{\lambda}
	<
	\bigg(\frac{1+\sqrt{1-\alpha}}{1-\sqrt{1-\alpha}}\bigg)^2,
\end{equation*}
and
\begin{equation*}
	\min_{Q\in O(n)}\trace\big\{A^{1/2}J_\alpha A^{1/2}(I-Q)\big\}
	<
	0
	\qquad\Longleftrightarrow\qquad
	\frac{\Lambda}{\lambda}
	\geq
	1
	\quad\text{ and }\quad
	\alpha\in(0,1).
\end{equation*}
\end{example}

\begin{example}
The mirror point coupling $Q=J_0$ fails to give \eqref{ineq-Q} when the distortion is large.
Let
\begin{equation*}
	A
	=
	\sqmatrix{5}{-12}{-12}{29},
\end{equation*}
which is a positive definite matrix and has the following square root:
\begin{equation*}
	A^{1/2}
	=
	\sqmatrix{1}{-2}{-2}{5}.
\end{equation*}
Then
\begin{equation*}
	A^{1/2}J_\alpha A^{1/2}
	=
	\sqmatrix{3}{-8}{-8}{21}+\alpha\sqmatrix{1}{-2}{-2}{4}
\end{equation*}
for every $\alpha\in(0,1)$. If we choose $Q=J_0$, then $I-Q=2e_1e_1^\top$, and thus
\begin{equation*}
	\trace\big\{A^{1/2}J_\alpha A^{1/2}(I-Q)\big\}
	=
	2e_1^\top A^{1/2}J_\alpha A^{1/2}e_1
	=
	2(3+\alpha)
	>
	0
\end{equation*}
for every $\alpha\in(0,1)$, so \eqref{n=2} (and thus \eqref{ineq-Q}) does not hold in general for $Q=J_0$.
On the other hand, for $Q$ the optimal coupling, \eqref{n=2} follows immediately from recalling \eqref{ctt A tr<0}.
\end{example}

\subsection{Counterexamples for large distortion}

The following examples show that the bound on the distortion $\Lambda/\lambda$ in the assumptions of our main theorem is really needed in this method. We first remark that this bound is only needed in the case $|x-z|>N\sqrt\lambda\,\varepsilon$. Indeed, when $|x-z|\leq N\sqrt\lambda\,\varepsilon$, the only place where the distortion crucially appears is in the inequality \eqref{intersection-ineq}, and for the proof in this case any bound on $\Lambda/\lambda$ is sufficient: for convenience we selected $\Lambda/\lambda\leq 3$ there. Thus, the only part of the paper where the bound on the distortion comes crucially into a play is in \Cref{Sec. >Nepsilon}, more precisely, in the condition \eqref{ineq-Q}.

As we explained in \Cref{Sec. Q}, given $A_1,A_2\in\A(\lambda,\Lambda)$ our strategy was based on showing \eqref{ineq-Q} for some $\alpha\in(0,1)$. Let us reformulate the problem by defining the ellipsoids $E_1:\,=E_{A_1}=A_1^{1/2}\B$ and $E_2:\,=E_{A_2}=A_2^{1/2}\B$ and letting $\phi(y)=A_2^{1/2}QA_1^{-1/2}y$. Then $\phi$  maps $E_1$ into $E_2$. Thus, performing a change of variables we see that \eqref{ineq-Q} is equivalent to
\begin{equation}\label{projection-ineq}
				\dashint_{E_1}|P_1(y-\phi(y))|^2\ dy
				>
				\dashint_{E_1}|(I-P_1)(y-\phi(y))|^2\ dy,
\end{equation}
where $P_ih$ stands for the projection of a vector $h\in\R^n$ over the $i$-th coordinate axis for $i=1,\ldots,n$, i.e.\ $P_ih=(e_i^\top h)e_i=(e_ie_i^\top) h$. That is, the method in \Cref{Sec. >Nepsilon} works if the length of the projection on $\{e_1\}^\bot$ of the difference between $y\in E_1$ and $\phi(y)\in E_2$ is larger (in average) than the length of such projection  over $\mathrm{span}\{e_1\}$.

Here, we can consider any measure preserving coupling map $\phi$ such that $\phi(E_1)=E_2$. This is the case in the following examples, which are constructed in such a way that the condition \eqref{projection-ineq} does not hold for any measure preserving map $\phi$. This is achieved by considering a pair of ellipsoids with large distortion oriented in such a way that the largest principal axes are aligned with coordinate axes orthogonal to $e_1$, then the difference $\phi(y)-y$ project over the first coordinate axis will necessarily be smaller (in average) than projected over one of the orthogonal axes to $e_1$. In other words, the average distance in terms of $f$ increases, no matter which coupling we use.

\subsubsection{A counterexample in two dimensions}\label{example-2D}

Let us consider an ellipse centered at $0$ with its largest axis oriented in the direction of $e_2$ and a ball,
\begin{equation*}
				E_1
				=
				\sqmatrix{1/10}{0}{0}{10}\B
				\quad \mbox{ and } \quad
				E_2
				=
				\sqmatrix{1}{0}{0}{1}\B.
\end{equation*}
In this case, both the ellipses have the same area, $|E_1|=|E_2|=\pi$, but different distortion: $E_1$ has distortion $\sqrt{100}$ while the distortion of $E_2$ is equal to $1$. Let $\phi:E_1\to E_2$ be any measure preserving map.
Let us denote $y=(y_1,y_2)^\top\in\R^2$. Since $|y_1|<\frac{1}{10}$ in $E_1$ and $|y_1|<1$ in $E_2$, it turns out that $|P_1(\phi(y)-y)|<\frac{11}{10}$ in $E_1$, and thus
\begin{equation*}
				\dashint_{E_1}|P_1(\phi(y)-y)|^2\ dy
				\leq
				1.21.
\end{equation*}
On the other hand, since $|y_2|<1$ in $E_2$, then $|P_2\phi(y)|< 1$ in $E_1$.  Moreover, since $E_1\cap\{|y_2|\geq 5\}$ is a non-empty subset of $E_1$ with positive measure, we can estimate $|P_2(\phi(y)-y)|$ in $E_1\cap\{|y_2|\geq 5\}$ from below by $4$. Indeed, by the triangle inequality,
\begin{equation*}
				|P_2(\phi(y)-y)|
				\geq
				\big||P_2\phi(y)|-|y_2|\big|
				>
				|y_2|-1
				\geq
				4
\end{equation*}
for every $y\in E_1$ such that $|y_2|\geq 5$.
Then, averaging we get
\begin{equation*}
\begin{split}
				\dashint_{E_1}|P_2(\phi(y)-y)|^2\ dy
				&
				\geq
				\frac{1}{|E_1|}\int_{E_1\cap\left\{|y_2|\geq 5\right\}}|P_2(\phi(y)-y)|^2\ dy
				\\
				&
				\geq
				16\,\frac{|E_1\cap\{\abs{y_2}\geq 5\}|}{|E_1|}.
\end{split}
\end{equation*}
Since $|E_1|=\pi$ and
\begin{equation*}
				E_1\cap\{|y_2|\geq 5\}
				=
				\sqmatrix{1/10}{0}{0}{10}\big(\B\cap\{|y_2|\geq 1/2\}\big),
\end{equation*}
then
\begin{equation*}
\begin{split}
				16\,\frac{|E_1\cap\{|y_2|\geq 5\}|}{|E_1|}
				=
				~&
				\frac{16}{\pi}\,|\B\cap\{|y_2|\geq 1/2\}|
				=
				\frac{64}{\pi}\int_{1/2}^1\sqrt{1-t^2}\ dt
				\\
				=
				~&
				16\bigg(\frac{2}{3}-\frac{\sqrt{3}}{2\pi}\bigg)
				\simeq
				6.26\ldots
\end{split}
\end{equation*}
Hence,
\begin{equation*}
				\dashint_{E_1}|P_1(\phi(y)-y)|^2\ dy
				<
				\dashint_{E_1}|P_2(\phi(y)-y)|^2\ dy,
\end{equation*}
which contradicts \eqref{projection-ineq}.

\subsubsection{A counterexample in three dimensions}\label{example-3D}
In this case we consider two ellipsoids with the same shape but with different orientations,
\begin{equation*}
				E_1
				=
				\begin{pmatrix}
				1 & 0 & 0 \\
				0 & 100 & 0 \\
				0 & 0 & 1
				\end{pmatrix}\B
				\quad \mbox{ and } \quad
				E_2
				=
				\begin{pmatrix}
				1 & 0 & 0 \\
				0 & 1 & 0 \\
				0 & 0 & 100
				\end{pmatrix}\B.
\end{equation*}
Observe that the largest axes of the ellipsoids are aligned with $e_2$ and $e_3$, respectively.
Let us denote $y=(y_1,y_2,y_3)^\top\in\R^3$. Then $|y_1|< 1$ in both $E_1$ and $E_2$, that is, for any measure preserving map $\phi:E_1\to E_2$, we have $|y_1|< 1$ and $|P_1\phi(y)|<1$, and thus
\begin{equation*}
				\dashint_{E_1}|P_1(\phi(y)-y)|^{2}\ dy
				\leq
				4.
\end{equation*} 
On the other hand, let us consider the set $E_1\cap\{|y_2|\geq 5\}$. Since $|y_2|<1$ in $E_2$, then the distance between a point in $E_2$ and $E_1\cap\{|y_2|\geq 5\}$ is bounded from below by $4$. In particular,  $|P_2(\phi(y)-y)|\geq |y_2|-|P_2\phi(y)|\geq 4$ in $E_1\cap\{|y_2|\geq 5\}$, and thus
\begin{equation*}
\begin{split}
				\dashint_{E_1} |(I-P_1)(\phi(y)-y)|^2\ dy
				&
				=
				\dashint_{E_1} |P_2(\phi(y)-y)|^2\ dy+\dashint_{E_1} |P_3(\phi(y)-y)|^2\ dy
				\\
				&
				\geq
				\dashint_{E_1} |P_2(\phi(y)-y)|^2\ dy
				\\
				&
				\geq
				\frac{1}{|E_1|}\int_{E_1\cap \{|y_2|\geq 5\}} |P_2(\phi(y)-y)|^2\ dy
				\\
				&
				\geq
				16\,\frac{|E_1 \cap  \{|y_2|\geq 5\}|}{|E_1|}
				\geq
				8.
\end{split}
\end{equation*}
The last inequality follows by the fact that the set $\{|y_2|\geq 5\}$ takes at least one half of the volume of the ellipsoid. Then
\begin{equation*}
				\dashint_{E_1}|P_1(\phi(y)-y)|^2\ dy
				<
				\dashint_{E_1}|(I-P_1)(\phi(y)-y)|^2\ dy,
\end{equation*}
which also contradicts \eqref{projection-ineq}.

\ 

\noindent \textbf{Acknowledgements.} The authors have been supported by the Academy of Finland project \#298641. \'A.~A. was partially supported by the grant MTM2017-85666-P. The authors would like to thank Eero Saksman for discussions related to this paper. \\ 

\def\cprime{$'$} \def\cprime{$'$} \def\cprime{$'$}

\end{document}